\documentclass[11pt]{article}

\usepackage[margin=1in,a4paper]{geometry}
\usepackage[T1]{fontenc}
\usepackage{lmodern}
\usepackage{amssymb,amsmath,amsthm,mathtools}
\usepackage{microtype}
\usepackage{graphicx}
\usepackage{placeins}
\usepackage{xcolor}
\definecolor{paperlink}{RGB}{0,0,114}
\usepackage[unicode,pagebackref,hypertexnames=false,breaklinks]{hyperref}
\hypersetup{colorlinks=true,linktoc=all,
  linkcolor=paperlink,citecolor=paperlink,urlcolor=paperlink,pdfstartview=Fit,
  pdftitle={Intermittency for the parabolic Anderson model on spaces with sub-Gaussian heat kernel estimates},
  pdfauthor={Fabrice Baudoin, Hongyi Chen, Li Chen, S{\o}ren Kj{\ae}rsgaard}}
\usepackage[shortlabels]{enumitem}
\setlist[enumerate,1]{wide, labelindent=0pt,label={\upshape(\roman*)}}
\setlist[itemize]{leftmargin=2em,itemsep=0.25em,topsep=0.35em}
\usepackage{fancyhdr}
\numberwithin{equation}{section}
\mathtoolsset{showonlyrefs}
\hbadness=\maxdimen

   \newtheorem{mainthm}{Theorem}
   
\newcommand{\R}{\mathbb{R}}

\newcommand{\EE}{\mathbb{E}}
\newcommand{\PP}{\mathbb{P}}
\newcommand{\eps}{\varepsilon}
\newcommand{\Beta}{\mathrm{B}}
\newcommand{\norm}[1]{{\left\lVert{#1}\right\rVert}}

\newcommand{\cE}{\mathcal{E}}
\newcommand{\cF}{\mathcal{F}}
\newcommand{\cH}{\mathcal{H}}
\newcommand{\cW}{\mathcal{W}}

\newcommand{\cD}{\mathcal{D}}

\newcommand{\1}{\mathbf{1}}
\newcommand{\Lap}{\Delta}

\theoremstyle{plain}
\newtheorem{thm}{Theorem}[section]
\newtheorem{lem}[thm]{Lemma}
\newtheorem{cor}[thm]{Corollary}
\newtheorem{prop}[thm]{Proposition}
\newtheorem{assump}[thm]{Assumption}
\theoremstyle{definition}
\newtheorem{defn}[thm]{Definition}
\newtheorem{notation}[thm]{Notation}
\theoremstyle{remark}
\newtheorem{remark}[thm]{Remark}
\newtheorem{example}[thm]{Example}

\begin{document}

\title{Intermittency for the parabolic Anderson model \\ on  spaces with sub-Gaussian heat kernel estimates}
\author{%
Fabrice Baudoin\thanks{E-mail: \href{mailto:fbaudoin@math.au.dk}{\nolinkurl{fbaudoin@math.au.dk}}},
Hongyi Chen\thanks{E-mail: \href{mailto:hchen77@math.au.dk}{\nolinkurl{hchen77@math.au.dk}}},
Li Chen\thanks{E-mail: \href{mailto:lchen@math.au.dk}{\nolinkurl{lchen@math.au.dk}}},
S\o ren Kj\ae rsgaard\thanks{E-mail: \href{mailto:skj@math.au.dk}{\nolinkurl{skj@math.au.dk}}}\\[0.4em]
{\small Department of Mathematics, Aarhus University}\\
{\small Ny Munkegade 118, 8000 Aarhus C, Denmark}}
\date{}
\maketitle

\begin{abstract}

We study the parabolic Anderson model $\partial_t u=\Delta u+\beta u\,\dot W_\alpha$ on an unbounded Ahlfors $d_h$-regular Dirichlet metric measure space whose heat kernel satisfies two-sided sub-Gaussian estimates with walk dimension $d_w$. The noise is white in time and, in space, either white ($\alpha=0$) or colored  through the Riesz kernel of $(-\Delta)^{-2\alpha}$, where $0\le 2\alpha<d_h/d_w<1+2\alpha$. No recurrence is assumed. We prove existence, uniqueness and H\"older regularity of the mild Skorohod solution for bounded initial conditions. For nonzero nonnegative bounded initial conditions, the solution is positive and its moment Lyapunov exponents exist for every real $p\ge1$, independently of the spatial point and the initial condition. For $p\ge2$, these exponents have order $p\,(\beta^2(p-1))^{d_w/((1+2\alpha)d_w-d_h)}$, with upper and lower bounds matching up to constants.  
\end{abstract}

\tableofcontents
\clearpage
\section{Introduction}\label{sec:intro}

The parabolic Anderson model (PAM) is the linear stochastic heat equation with multiplicative Gaussian noise
\begin{equation}\label{eq:pam-intro}
\partial_{t}u(t,x)=\Lap u(t,x)+\beta\,u(t,x)\,\dot W_\alpha(t,x),\qquad u(0,\cdot)=u_0\,,
\end{equation}
where $\beta>0$ is the coupling constant, physically an inverse temperature. A central feature of the model is intermittency: the solution develops in time increasingly rare and higher peaks. When the moment Lyapunov exponents $\lambda_p=\lim_{t\to\infty}t^{-1}\log\EE[u(t,x)^p]$ exist, the intermittency property is expressed by the strict increase of $p\mapsto\lambda_p/p$~\cite{CarmonaMolchanov1994,Khoshnevisan2014}. The dependence of the moment Lyapunov exponents on $\beta$ and $p$ quantifies the growth of these peaks. We study~\eqref{eq:pam-intro} on unbounded metric measure spaces equipped with a diffusion satisfying two-sided sub-Gaussian heat kernel estimates. The noise is white in time and either white or Riesz kernel colored in space. Our main results establish well-posedness and H\"older regularity, the existence of moment Lyapunov exponents for every real $p\ge1$, and two-sided bounds with matching dependence on $\beta$ and $p$ for $p\ge2$.

\vspace{0.4cm}

On $\R^d$, the model and its intermittency properties have been studied extensively since the seminal celebrated work of Carmona and Molchanov~\cite{CarmonaMolchanov1994}. For noise that is white in time and homogeneous in space, the stochastic integration theory of Walsh and Dalang~\cite{Walsh1986,Dalang1999} provides the natural framework, with admissible spatial covariances characterized by Dalang's condition. Well-posedness, Feynman--Kac representations, moment estimates and regularity for Gaussian multiplicative noise are developed in~\cite{HuNualart2009,HuNualartSong2011,HHNT15,BalanChen2018,ChenHuang2019}. Intermittency is studied in~\cite{FoondunKhoshnevisan2009,ChenDalang2015,Khoshnevisan2014}, and precise long-time moment asymptotics are obtained in~\cite{ChenHuSongXing2015,HuangLeNualart2017,Ch15,Ch17}. Spatial homogeneity plays an important role in this theory: Fourier analysis expresses Dalang's condition in terms of the spectral measure of the covariance and provides estimates for the Wiener chaos kernels, while Feynman--Kac formulas relate integer moments to interacting particle systems.

Beyond Euclidean space, the model has more recently been studied on Heisenberg groups~\cite{BOTW23,Heisenberg_Stratonovich}, compact Riemannian manifolds~\cite{ChenOuyangVickery2025,CO25,CNO26}, and hyperbolic spaces~\cite{BCO24,Geng-Ouyang}; see also~\cite{GengXu2025,GWX_quenched} for time-independent noise. Closest to the present work are the results on compact p.c.f.\ fractals~\cite{HamblyYang2018,HamblyYang2026} and bounded domains of recurrent metric measure spaces~\cite{BCHOTW24}. A unified account of the Euclidean and geometric theories is given in~\cite{TindelAarhus2026}. In these settings, heat kernel theory provide an alternative to Fourier analysis. Here we develop this approach on unbounded spaces under two-sided sub-Gaussian bounds, without assuming recurrence or spatial homogeneity. The framework includes Euclidean spaces, Carnot groups, complete Riemannian manifolds with nonnegative Ricci curvature and maximal volume growth, and unbounded fractals such as the infinite Sierpi\'nski gasket, nested fractals and Sierpi\'nski carpets.

\paragraph{Setting.}
Let $(X,d,\mu)$ be a locally compact, complete, unbounded, geodesic metric measure space, and let $(\cE,\cF)$ be a strongly local regular Dirichlet form on $L^2(X,\mu)$ with generator $\Lap$ and semigroup $P_t=e^{t\Lap}$. Assume that $\mu$ is Ahlfors $d_h$-regular, so that $\mu(B(x,r))\simeq r^{d_h}$ for all $x\in X$ and $r>0$. We also assume that $P_t$  admits a jointly continuous heat kernel $p_t(x,y)$ satisfying two-sided sub-Gaussian bounds
\begin{equation}\label{eq:subGauss-intro}
p_t(x,y)\simeq t^{-\frac{d_h}{d_w}}\exp\biggl(-c\Bigl(\frac{d(x,y)^{d_w}}{t}\Bigr)^{\frac1{d_w-1}}\biggr),\qquad t>0,\ x,y\in X\,,
\end{equation}
for some $d_w\ge2$. The Hunt process $(B_t)_{t\ge0}$ associated with $(\cE,\cF)$ is a diffusion, which we call Brownian motion on $X$. We assume that it is defined from every starting point. The precise hypotheses appear in Assumption~\ref{hyp:dirichlet-form}. We refer to~\cite{GrigoryanTelcs2012,GrigoryanHuLau2014} for a general account on sub-Gaussian heat kernel estimates on metric measure spaces and to~\cite{Barlow1998,Kigami2001,BarlowBass1999} for the fractal examples which provide a motivation for this work.

The exponents $d_h$ and $d_w$ describe volume growth and diffusion scaling, respectively. Ahlfors regularity identifies $d_h$ with the Hausdorff dimension, while the walk dimension $d_w$ means that Brownian motion travels a distance of order $r$ in a time of order $r^{d_w}$. Since $p_t(x,x)\simeq t^{-d_h/d_w}$, the diffusion is transient precisely when $d_h>d_w$. Our assumptions allow both the strongly recurrent regime $d_h<d_w$, treated for bounded domains in~\cite{BCHOTW24}, and the critical and transient regimes $d_h\ge d_w$.

\paragraph{The noise.}
For $\alpha>0$, the spatial covariance  we chose for the driving noise is the Riesz kernel of $(-\Lap)^{-2\alpha}$:
\begin{equation}\label{eq:Riesz-intro}
G_{2\alpha}(x,y)=\frac1{\Gamma(2\alpha)}\int_0^\infty t^{2\alpha-1}p_t(x,y)\,dt\,.
\end{equation}
By~\eqref{eq:subGauss-intro}, the integral is finite off the diagonal precisely when $2\alpha<d_h/d_w$. In that range, $G_{2\alpha}(x,y)\simeq d(x,y)^{2\alpha d_w-d_h}$; see Lemma~\ref{lem:fractional-laplacian-kernel-existence-and-bound}. On $\R^d$, this is a constant multiple of the Riesz kernel $|x-y|^{-(d-4\alpha)}$. We take $W_\alpha$ to be the centered Gaussian noise with covariance
\begin{equation*}
\EE\bigl[W_\alpha(\varphi)W_\alpha(\psi)\bigr]=\int_0^\infty\!\!\int_X\!\int_X\varphi(t,x)\psi(t,y)\,G_{2\alpha}(x,y)\,d\mu(x)\,d\mu(y)\,dt\,.
\end{equation*}
Formally, $\dot W_\alpha=(-\Lap)^{-\alpha}\dot W_0$: the noise is white in time and  smoothed in space by a fractional power of the Laplacian. These types of fields are introduced and studied on fractals in~\cite{BaudoinLacaux2022,BaudoinChen2023}. For $\alpha=0$, the noise is white in both time and space, with spatial covariance given by the $L^2(X,\mu)$ inner product. 

\paragraph{The admissible range of $\alpha$.}
Throughout the paper,
\begin{equation}\label{eq:alpha-range-intro}
0\le\alpha<\frac{d_h}{2d_w},\qquad \alpha>\frac{d_h}{2d_w}-\frac12\,.
\end{equation}
For $\alpha>0$, the upper bound ensures that $G_{2\alpha}$ is finite off the diagonal. The lower bound is the analogue of Dalang's condition: see Lemmas~\ref{lem:sup-kernel-estimates} and~\ref{lem:sobolev-IP}. As we will see, the quantity $(1+2\alpha)d_w-d_h\in(0,d_w)$ governs both moment growth and H\"older regularity of the solution. Space--time white noise, corresponding to $\alpha=0$, is admissible exactly in the strongly recurrent case $d_h<d_w$; in Euclidean space, this restricts the dimension to $d=1$.

\paragraph{Solutions.}
For $u_0\in L^\infty(X,\mu)$, we interpret~\eqref{eq:pam-intro} in the mild It\^o--Skorohod sense. Thus a solution is an adapted random field satisfying, for every $(t,x)\in\R_+\times X$,
\begin{equation}\label{eq:mild-intro}
u(t,x)=P_tu_0(x)+\beta\int_0^t\!\!\int_Xp_{t-s}(x,y)\,u(s,y)\,W_\alpha(ds,dy)\qquad\PP\text{-a.s.}\,,
\end{equation}
where the stochastic integral is understood in the Walsh--Dalang sense; see Definition~\ref{def:mild-PAM-solution}.

\paragraph{Main results.} We first state the well-posedness and regularity results.
\begin{mainthm}\label{thm:main-existence}
Suppose that Assumption~\ref{hyp:dirichlet-form} holds, $\alpha$ satisfies~\eqref{eq:alpha-range-intro}, $\beta>0$, and $u_0\in L^\infty(X,\mu)$. Then~\eqref{eq:pam-intro} has a unique mild solution $u$ in the sense of Definition~\ref{def:mild-PAM-solution}. It is given by its Wiener chaos expansion, and for every real $p\ge2$ and every $T>0$,
\begin{equation*}
\sup_{(t,x)\in(0,T]\times X}\EE\bigl[|u(t,x)|^p\bigr]<\infty\,.
\end{equation*}
If $u_0\ge0$ $\mu$-a.e., then $u(t,x)\ge0$ a.s.\ for every $(t,x)\in\R_+\times X$. If, in addition, $u_0\ne0$ in $L^\infty(X,\mu)$, then $u(t,x)>0$ a.s.\ for every such $(t,x)$.
\end{mainthm}

Theorem~\ref{thm:main-existence} follows from Theorem~\ref{thm:exist-unique-wholespace} and Proposition~\ref{prop:Lp-estimates}. The next result gives the local H\"older exponents in terms of the heat semigroup regularity.

\begin{mainthm}\label{thm:main-holder}
Let Assumption~\ref{hyp:dirichlet-form} hold and suppose that, for some $\kappa\in(0,1]$, the semigroup satisfies the weak Bakry--\'Emery condition $\mathrm{wBE}(\kappa)$ of~\cite{Besov3}: there is $C>0$ such that
\begin{equation}\label{eq:wBE-intro}
|P_tg(x)-P_tg(y)|\le C\,\frac{d(x,y)^\kappa}{t^{\kappa/d_w}}\,\|g\|_{L^\infty},\qquad t>0,\ g\in L^\infty(X,\mu),\ x,y\in X\,.
\end{equation}
Let $\alpha$ satisfy~\eqref{eq:alpha-range-intro} and let $\beta>0$. For $u_0\in L^\infty(X,\mu)$, let $u$ be the solution of~\eqref{eq:pam-intro}. Then $u$ admits a jointly continuous modification on $(0,\infty)\times X$. On every set $[T,T']\times B(x_0,R)$ with $0<T<T'<\infty$, $x_0\in X$ and $R>0$, this modification is $\gamma_t$-H\"older continuous in time and $\gamma_x$-H\"older continuous in space for any
\begin{equation*}
0<\gamma_t<\frac{(1+2\alpha)d_w-d_h}{2d_w},\qquad 0<\gamma_x<\min\Bigl\{\kappa,\frac{(1+2\alpha)d_w-d_h}{2}\Bigr\}\,.
\end{equation*}
\end{mainthm}

Theorem~\ref{thm:main-holder}, including moment estimates for the increments, is proved as Theorem~\ref{thm:Holder}. Assumption~\ref{hyp:dirichlet-form} already implies~\eqref{eq:wBE-intro} for some $\kappa\in(0,1)$, by the parabolic Harnack inequality and the H\"older regularity of caloric functions~\cite{BarlowGrigoryanKumagai2012}. The optimal value of $\kappa$, however, reflects finer geometric information. One may take $\kappa=1$ on complete Riemannian manifolds with nonnegative Ricci curvature, while on the Sierpi\'nski gasket the optimal exponent is $\kappa=d_w-d_h$~\cite{Besov3}; see also~\cite{Besov1,Besov3} for the role of this condition in the analysis of Dirichlet spaces. On $\R^d$, the limiting exponents are $(2+4\alpha-d)/4$ in time and $(2+4\alpha-d)/2$ in space, in agreement with~\cite[Theorem~1.8]{ChenHuang2019}. For the Sierpi\'nski gasket with white noise, they are $(d_w-d_h)/(2d_w)$ and $(d_w-d_h)/2$, respectively. Expressed in the metric $d$, these agree with the compact-fractal exponents in~\cite{HamblyYang2018,HamblyYang2026}.

\begin{mainthm}\label{thm:main-lyapunov}
Let Assumption~\ref{hyp:dirichlet-form} hold, let $\alpha$ satisfy~\eqref{eq:alpha-range-intro}, and let $\beta>0$. For every real $p\ge1$, every nonzero nonnegative $u_0\in L^\infty(X,\mu)$ and every $x\in X$, the finite limit
\begin{equation}\label{eq:lyapunov-intro}
\lambda_p(\beta):=\lim_{t\to\infty}\frac1t\log\EE\bigl[u(t,x)^p\bigr]
\end{equation}
exists and depends neither on $x$ nor on $u_0$.
\end{mainthm}

Theorem~\ref{thm:main-lyapunov} is proved as Theorem~\ref{thm:lyapunov-main}. In the one-dimensional space--time white-noise setting, exact Lyapunov exponents at every positive real order were obtained by Das and Tsai~\cite{DasTsai2021} for a Dirac initial condition and by Ghosal and Lin~\cite{GhosalLin2023} for a broad class of initial conditions. Our result establishes existence and independence of the spatial point and initial condition, for every real $p\ge1$. It applies to both colored noise and the admissible white-noise case. The next theorem gives quantitative bounds on these exponents.

The constants $c,C>0$ below depend only on $\alpha$ and the data in Assumption~\ref{hyp:dirichlet-form}; in particular, they are independent of $\beta$ and $p$. The dependence on $\beta$ agrees with the exact discrete scaling on the infinite Sierpi\'nski gasket; see Proposition~\ref{prop:sg-solution-scaling} and Remark~\ref{rem:lyapunov-scaling}.

\begin{mainthm}\label{thm:main-moments}
Let Assumption~\ref{hyp:dirichlet-form} hold, let $\alpha$ satisfy~\eqref{eq:alpha-range-intro} and let $\beta>0$. Let $\lambda_p(\beta)$ be the exponents~\eqref{eq:lyapunov-intro} associated with nonzero nonnegative bounded initial conditions.
\begin{enumerate}
\item For every real $p\ge2$,
\begin{equation}\label{eq:UB-intro}
\lambda_p(\beta)\le C\,p\,\bigl(\beta^2(p-1)\bigr)^{\frac{d_w}{(1+2\alpha)d_w-d_h}}\,.
\end{equation}
\item For every real $p>1$,
\begin{equation}\label{eq:LB-intro}
\lambda_p(\beta)\ge c\,p\,\bigl(\beta^2(p-1)\bigr)^{\frac{d_w}{(1+2\alpha)d_w-d_h}}\,.
\end{equation}
\end{enumerate}
Moreover, $\lambda_1(\beta)=0$, the function $p\mapsto\lambda_p(\beta)/p$ is strictly increasing on $(1,\infty)$, and
\begin{equation}\label{eq:intermittency-intro}
\lim_{p\to\infty}\frac{\lambda_p(\beta)}p=\infty\,.
\end{equation}
\end{mainthm}
For $1<p<2$, H\"older interpolation between the first and second moments, together with $\lambda_1(\beta)=0$, gives
\begin{equation*}
\lambda_p(\beta)\le(p-1)\lambda_2(\beta)\le 2C(p-1)\,\beta^{\frac{2d_w}{(1+2\alpha)d_w-d_h}}\,.
\end{equation*}
The upper bound~\eqref{eq:UB-intro} does not, in general, extend to $1<p<2$ with a constant uniform in $p$. Indeed, for $X=\R$, $\alpha=0$ and $u_0\equiv1$, Theorem~1.1 of~\cite{GhosalLin2023}, rescaled to the generator $\Delta$, gives
\begin{equation*}
\lambda_p(\beta)=\frac{\beta^4}{48}(p^3-p)\sim\frac{\beta^4}{24}(p-1)
\qquad\text{as }p\downarrow1\,.
\end{equation*}
Since $d_h=1$ and $d_w=2$, the right-hand side of~\eqref{eq:UB-intro} has order $\beta^4p(p-1)^2$, which cannot bound this linear behavior near $p=1$. The interpolation estimate therefore has the correct order in $p-1$ in this example.

Theorem~\ref{thm:main-moments} follows from Theorem~\ref{thm:lyapunov-main}, Proposition~\ref{prop:Lp-estimates} and Theorem~\ref{thm:moment-LB}; see Corollary~\ref{cor:lyapunov-bounds}. The upper moment bound also holds for arbitrary bounded initial conditions. Strict increase of $p\mapsto\lambda_p(\beta)/p$, together with~\eqref{eq:intermittency-intro}, gives full intermittency in the sense of~\cite[Section~1.2]{TindelAarhus2026}. For $p\ge2$, the bounds have matching dependence on $\beta$ and $p$. In particular, on $\R$ with white noise they have order $\beta^4p^3$, consistent with the exact asymptotics in~\cite{Ch15}. The comparison with the compact-fractal moment bounds of~\cite{HamblyYang2026} appears in Remark~\ref{rem:compact-moment-comparison}.

\paragraph{Proof strategy and relation to the literature.}
Existence, uniqueness and upper moment bounds follow from Wiener chaos expansions and heat kernel estimates, extending the approach of~\cite{BCHOTW24}. For H\"older continuity, we use the weak Bakry--\'Emery estimates developed in~\cite{Besov3}. To prove existence of the Lyapunov exponents, we first construct the fundamental solution and its stochastic flow, extending the framework of Alberts, Khanin and Quastel~\cite{AlbertsKhaninQuastel2014}. We then adapt the moment comparisons of Berger, Chong and Lacoin~\cite[Lemma~5.1 and proof of Proposition~2.8]{BergerChongLacoin2023}.The argument applies to every real $p\ge1$ without spatial translation invariance. Finally, the Gaussian exponential test of~\cite[Proposition~3.20]{ChenDeyaOuyangTindel2021}, combined with confinement of the diffusion to a ball, gives lower bounds for every real $p>1$.

\paragraph{Organization.}
Section~\ref{sec:existence} develops the analytic preliminaries, constructs the noise and the solution, and proves well-posedness, positivity, finite-time moment bounds and H\"older regularity. Section~\ref{sec:lyapunov} constructs the fundamental solution and proves the existence and universality of the moment Lyapunov exponents. Section~\ref{sec:moment-bounds} establishes their two-sided bounds and derives intermittency. Section~\ref{sec:infinite-gasket} illustrates these results in the infinite Sierpi\'nski gasket and proves the discrete scaling law of the solution.

\paragraph{Conventions.}
Throughout the paper, $C,c>0$ denote constants whose values may change from line to line. Unless otherwise stated, they depend only on $\alpha$ and the data in Assumption~\ref{hyp:dirichlet-form}. We write $\R_+=(0,\infty)$ and $B(x,r)=\{y\in X:d(x,y)<r\}$.

\section{Existence, uniqueness and H\"older regularity of the solution}\label{sec:existence}

We first start with some needed preliminaries on Dirichlet spaces. We then construct the noise, and obtain the PAM solution from its Wiener chaos expansion. The chaos estimates yield existence, uniqueness and finite-time $L^p$ bounds. We show that the solution is non-negative if the initial condition is. We finally prove H\"older regularity by using weak Bakry-\'Emery type estimates \cite{Besov3} for the heat semigroup.

\subsection{Dirichlet space, heat kernel and exit time estimates}\label{sec:setting}\label{sec:heat-kernel}

Throughout the paper we work under the following standing assumption.

\begin{assump}\label{hyp:dirichlet-form}
Let $(X,d,\mu)$ be a locally compact, complete, unbounded, geodesic metric measure space, and let $(\cE,\cF)$ be a Dirichlet form on $L^2(X,\mu)$ with associated semigroup $\{P_t\}_{t\ge0}$ and nonpositive self-adjoint generator $\Lap$. We assume:
\begin{enumerate}
\item The measure $\mu$ is Ahlfors $d_h$-regular: there are $d_h>0$ and constants $c,C>0$ such that, for all $x\in X$ and $r>0$,
\begin{equation*}
cr^{d_h}\le\mu\bigl(B(x,r)\bigr)\le Cr^{d_h}\,.
\end{equation*}
\item $\cE$ is regular and strongly local. 
\item The semigroup is stochastically complete: $P_t1=1$ for every $t\ge0$.
\item $\{P_t\}_{t\ge0}$ admits a heat kernel $p_t(x,y)$, jointly continuous on $(0,\infty)\times X\times X$, satisfying two-sided sub-Gaussian estimates: there are $d_w\ge2$ and constants $c,C>0$ such that, for all $t>0$ and $x,y\in X$,
\begin{equation}\label{eq:subGauss}
ct^{-d_h/d_w}\exp\biggl(-C\Bigl(\frac{d(x,y)^{d_w}}{t}\Bigr)^{\frac{1}{d_w-1}}\biggr)
\le p_{t}(x,y)\le Ct^{-d_h/d_w}\exp\biggl(-c\Bigl(\frac{d(x,y)^{d_w}}{t}\Bigr)^{\frac{1}{d_w-1}}\biggr)\,.
\end{equation}
\item The Hunt process $(\Omega,(B_t)_{t\ge0},(\PP_x)_{x\in X})$ associated with $(\cE,\cF)$ is defined for every starting point $x\in X$, has transition density $p_t$ with respect to $\mu$, i.e.\ $\PP_x(B_t\in A)=\int_Ap_t(x,y)\,d\mu(y)$ for every Borel set $A$, and is a strong Markov process with continuous sample paths.
\end{enumerate}
\end{assump}

Unlike \cite{BCHOTW24} we do \emph{not} assume $d_h<d_w$: the Hunt process may be transient. Since $p_t$ is continuous and $\mu$ has full support, the estimates~\eqref{eq:subGauss} hold for all $(x,y)$ as soon as they hold for $\mu\otimes\mu$-almost every $(x,y)$. The Chapman-Kolmogorov identity
\begin{equation}\label{eq:chapman}
p_{s+t}(x,y)=\int_X p_s(x,z)p_t(z,y)\,d\mu(z)\qquad\text{and}\qquad \int_X p_t(x,y)\,d\mu(y)=1
\end{equation}
hold for every $s,t>0$ and $x,y\in X$. Item (v) is automatic in the examples we have in mind. Indeed, the strong locality of $\cE$ means that the Hunt process is a diffusion~\cite[Section 4.5]{FOT2011}, and the continuity of $p_t$ together with the upper bound in~\eqref{eq:subGauss} imply that $P_t$ maps bounded measurable functions to bounded continuous functions, so that the process can be started from every point; see~\cite{Barlow1998,Kigami2001,BarlowBass1999} for the fractal examples, \cite{BBK2006,GrigoryanTelcs2012} for the general theory, and~\cite[Section 2]{BCHOTW24}. We refer to $(B_t)_{t\ge0}$ as the Brownian motion on $X$, and we write $\EE_x$ for the expectation under $\PP_x$.

\begin{example}
Assumption~\ref{hyp:dirichlet-form} is satisfied by $\R^d$ with Lebesgue measure and the standard Dirichlet form, Carnot groups of homogeneous dimension $Q$, and complete $d$-dimensional Riemannian manifolds with nonnegative Ricci curvature and maximal volume growth. In these cases $d_w=2$ and $d_h=d,Q,d$, respectively. It also holds on unbounded nested fractals~\cite{Barlow1998,Kigami2001}, including the infinite Sierpi\'nski gasket~\cite{BarlowPerkins1988}, for which $d_h=\log3/\log2$ and $d_w=\log5/\log2$, and on unbounded generalized Sierpi\'nski carpets~\cite{BarlowBass1999}. In the last two families $d_w>2$, and both the strictly recurrent regime $d_h<d_w$ and the regime $d_h\ge d_w$ occur.
\end{example}

\begin{notation}\label{not:parameters}
Throughout the paper, unless otherwise stated, we consider a parameter $\alpha$ which satisfies
\begin{equation}\label{eq:alpha-range}
0\le\alpha<\frac{d_h}{2d_w},\qquad \alpha>\frac{d_h}{2d_w}-\frac12\,,
\end{equation}
so that
\begin{equation}\label{eq:parameter-inequalities}
0<1+2\alpha-\frac{d_h}{d_w}<1,\qquad 0<d_h-2\alpha d_w<d_w\,.
\end{equation}
Thus $\alpha=0$ is admissible if and only if $d_h<d_w$.
\end{notation}

Symmetry and the semigroup identity~\eqref{eq:chapman} give
\begin{equation}\label{eq:heatkernel-square-bound}
\int_Xp_t(x,y)^2\,d\mu(y)=p_{2t}(x,x)\le Ct^{-\frac{d_h}{d_w}},
\qquad t>0,\ x\in X\,.
\end{equation}
Positivity and stochastic completeness also give the contraction bound
\begin{equation}\label{eq:semigroup-infinity-bound}
\norm{P_tu_0}_{L^\infty}\le\norm{u_0}_{L^\infty},
\qquad t\ge0,\ u_0\in L^\infty(X,\mu)\,.
\end{equation}

The next lemma converts the off-diagonal heat kernel bound into a tail estimate for the diffusion. Such estimates are classical; see~\cite[Section~3]{Barlow1998} and~\cite{GrigoryanTelcs2012}. We include the short proof.

\begin{lem}\label{lem:tail-exit-UB}
There exist constants $C,c>0$ such that for all $z\in X$, $s>0$ and $R\ge s^{1/d_w}$,
\begin{equation}\label{eq:oneParticleTail}
\PP_z\bigl(d(B_s,z)\ge R\bigr)\le C\exp\Bigl(-c\Bigl(\frac{R^{d_w}}{s}\Bigr)^{\frac1{d_w-1}}\Bigr)\,.
\end{equation}
\end{lem}

\begin{proof}
Set $\lambda=(R^{d_w}/s)^{1/(d_w-1)}\ge1$ and $\gamma=d_w/(d_w-1)$. Decomposing $\{d(\cdot,z)\ge R\}$ into the annuli $A_n=B(z,2^{n+1}R)\setminus B(z,2^nR)$, $n\ge0$, and using the upper bound in~\eqref{eq:subGauss} together with the Ahlfors regularity, we get
\begin{align*}
\PP_z\bigl(d(B_s,z)\ge R\bigr)&\le\sum_{n\ge0}Cs^{-d_h/d_w}e^{-c2^{n\gamma}\lambda}\mu\bigl(B(z,2^{n+1}R)\bigr)
\\
&\le C\Bigl(\frac{R^{d_w}}{s}\Bigr)^{d_h/d_w}\sum_{n\ge0}2^{nd_h}e^{-c2^{n}\lambda}\,.
\end{align*}
For $\lambda\ge1$, the series is bounded by $Ce^{-c\lambda}$. Since $(R^{d_w}/s)^{d_h/d_w}=\lambda^{d_h(d_w-1)/d_w}$, the polynomial factor can be absorbed by decreasing $c$. This proves~\eqref{eq:oneParticleTail}.

\end{proof}

For a ball $B\subset X$, let $\tau_B=\inf\{t>0:B_t\notin B\}$ be its exit time. The following lower bound on the probability of staying in a ball is the key ingredient for the moment lower bounds of Subsection~\ref{sec:lower}.

\begin{lem}\label{lem:ball-exit-time-LB}
There are constants $\delta_0\in(0,1)$ and $C>0$, depending only on the geometric data, such that, for all $x\in X$, $r>0$ and $t>0$,
\begin{equation}\label{eq:exit-LB}
\PP_x\bigl(\tau_{B(x,r)}>t\bigr)\ge\delta_0\exp\Bigl(-C\frac{t}{r^{d_w}}\Bigr)\,.
\end{equation}
\end{lem}
\begin{proof}
Fix $x\in X$ and $r>0$, and write $B=B(x,r)$ and $\tau=\tau_B$. Ahlfors regularity gives volume doubling and, together with completeness, implies that all balls are relatively compact. By~\cite[Theorem 3.1 and (3.10)]{BarlowGrigoryanKumagai2012}, applied with the rate function $s\mapsto s^{d_w}$, the estimates~\eqref{eq:subGauss} imply that there exist $\eps\in(0,1)$ and $c>0$, depending only on the structural constants, such that the killed heat kernel satisfies
\begin{equation*}
p_{(\eps r)^{d_w}}^B(y,z)\ge c(\eps r)^{-d_h},
\qquad y,z\in B(x,\eps^2r)\,.
\end{equation*}
The continuity assertion in that theorem extends the bound from almost every pair $(y,z)$ to every pair. Integrating over $B(x,\eps^2r)$ and using Ahlfors regularity gives a constant $\delta_0\in(0,1)$, independent of $x$ and $r$, such that
\begin{equation}\label{eq:one-step}
\inf_{y\in B(x,\eps^2r)}
\PP_y\bigl(\tau>(\eps r)^{d_w},\
B_{(\eps r)^{d_w}}\in B(x,\eps^2r)\bigr)
\ge\delta_0\,.
\end{equation}

Since $x\in B(x,\eps^2r)$, iterating the Markov property at the times $k(\eps r)^{d_w}$ and using~\eqref{eq:one-step} yields, for every integer $n\ge1$,
\begin{equation*}
\PP_x\bigl(\tau>n(\eps r)^{d_w},\
B_{n(\eps r)^{d_w}}\in B(x,\eps^2r)\bigr)
\ge\delta_0^n\,.
\end{equation*}
For $t>0$, choose $n=\lceil t/(\eps r)^{d_w}\rceil
\le t/(\eps r)^{d_w}+1$. Then
\begin{align*}
\PP_x(\tau>t)
&\ge\PP_x\bigl(\tau>n(\eps r)^{d_w}\bigr)
\ge\delta_0^n\\
&\ge\delta_0\exp\Bigl(
-\eps^{-d_w}\log(1/\delta_0)\frac{t}{r^{d_w}}
\Bigr)\,,
\end{align*}
This proves~\eqref{eq:exit-LB}.
\end{proof}

\subsection{Negative powers of the Laplacian and Riesz kernels}\label{sec:Preliminaries}

We construct the negative powers of $-\Lap$, identify their Riesz kernels, and prove the kernel bounds and convolution identity needed to define the covariance of the noise we are interested in. Fractional Gaussian fields defined through powers of a fractal Laplacian were studied in~\cite{BaudoinLacaux2022,BaudoinChen2023}, but our non-compact framework introduces some new difficulty we need to handle.

\begin{defn}\label{def:fractional-laplacian}
For $\gamma>0$, the operator $(-\Lap)^{-\gamma}$ is defined by the spectral theorem: if $-\Lap=\int_{[0,\infty)}\lambda\,dE_\lambda$, then
\begin{equation*}
(-\Lap)^{-\gamma}=\int_{(0,\infty)}\lambda^{-\gamma}\,dE_\lambda,\qquad
\cD\bigl((-\Lap)^{-\gamma}\bigr)=\Bigl\{f\in L^2(X,\mu):\int_{(0,\infty)}\lambda^{-2\gamma}\,d\langle E_\lambda f,f\rangle<\infty\Bigr\}\,.
\end{equation*}
\end{defn}

Here $E_{\{0\}}=0$, so that the spectral integral may be restricted to $(0,\infty)$: by Minkowski's inequality and~\eqref{eq:heatkernel-square-bound}, $\|P_tf\|_{L^2}\le Ct^{-d_h/(2d_w)}\|f\|_{L^1}\to0$ for $f\in L^1\cap L^2$, hence $P_tf\to0$ in $L^2$ for all $f\in L^2$ by density, and $E_{\{0\}}$ is the strong limit of $P_t$ as $t\to\infty$. In particular $\ker\Lap=\{0\}$; note that, $X$ being unbounded and $\mu$ Ahlfors regular, no nonzero constant belongs to $L^2$. For  $\gamma=0$ we use the convention $(-\Lap)^0=I$ on $L^2(X,\mu)$.

Using $P_t=\int_{(0,\infty)}e^{-t\lambda}\,dE_\lambda$ and the integral $\lambda^{-\gamma}=\frac1{\Gamma(\gamma)}\int_0^\infty s^{\gamma-1}e^{-s\lambda}\,ds$, the spectral theorem gives the following integral representation of $(-\Lap)^{-\gamma}$.

\begin{lem}\label{lem:fraclap-integral}
Let $\gamma>0$ and $f\in L^2(X,\mu)$ satisfy $\int_0^\infty s^{\gamma-1}\|P_sf\|_{L^2}\,ds<\infty$. Then $f\in\cD((-\Lap)^{-\gamma})$ and
\begin{equation}\label{eq:fraclap-integral}
(-\Lap)^{-\gamma}f=\frac1{\Gamma(\gamma)}\int_0^\infty s^{\gamma-1}P_sf\,ds.
\end{equation}
\end{lem}

The following domain inclusion and smoothing estimate will be used to identify the Riesz kernels with the spectral powers.

\begin{lem}\label{lem:Schwarz-properties}
Let $0<\gamma<\frac{d_h}{2d_w}$. Then $L^1(X,\mu)\cap L^2(X,\mu)\subset\cD((-\Lap)^{-\gamma})$. Moreover, for every $f\in L^1(X,\mu)\cap L^2(X,\mu)$ and $t>0$,
\begin{equation*}
\norm{(-\Lap)^{-\gamma}P_tf}_{L^2}\le C\,t^{\gamma-\frac{d_h}{2d_w}}\norm{f}_{L^1}\,.
\end{equation*}
\end{lem}

\begin{proof}
By~\eqref{eq:heatkernel-square-bound} and Minkowski's inequality, $\|P_sf\|_{L^2}\le Cs^{-d_h/(2d_w)}\|f\|_{L^1}$. Together with $L^2$ contractivity, this yields
\begin{equation*}
\int_0^\infty s^{\gamma-1}\|P_sf\|_{L^2}\,ds
\le\frac{\|f\|_{L^2}}{\gamma}
+C\|f\|_{L^1}\int_1^\infty s^{\gamma-1-\frac{d_h}{2d_w}}\,ds<\infty\,.
\end{equation*}
Lemma~\ref{lem:fraclap-integral} proves the domain assertion. Applying the same estimate to $P_tf$ gives
\begin{equation*}
\|(-\Lap)^{-\gamma}P_tf\|_{L^2}
\le\frac{C\|f\|_{L^1}}{\Gamma(\gamma)}
\int_0^\infty s^{\gamma-1}(s+t)^{-\frac{d_h}{2d_w}}\,ds
=C_\gamma\|f\|_{L^1}t^{\gamma-\frac{d_h}{2d_w}}\,.
\end{equation*}
\end{proof}

\begin{defn}\label{def:fractional-laplacian-kernel}
For $\gamma\in(0,\frac{d_h}{d_w})$ and $x\ne y$, we set
\begin{equation}\label{eq:RieszKernel}
G_\gamma(x,y)=\frac{1}{\Gamma(\gamma)}\int_0^\infty t^{\gamma-1}p_t(x,y)\,dt\,.
\end{equation}
\end{defn}

The time integral in~\eqref{eq:RieszKernel} converges at $t=\infty$ because of the on-diagonal decay $t^{-d_h/d_w}$ and $\gamma<d_h/d_w$, and at $t=0$ because of the off-diagonal Gaussian factor.

\begin{lem}\label{lem:fractional-laplacian-kernel-existence-and-bound}
Let $\gamma\in(0,\frac{d_h}{d_w})$. There exist constants $c,C>0$ such that for all $x\ne y$,
\begin{equation}\label{eq:Riesz-bounds}
c\,d(x,y)^{\gamma d_w-d_h}\le G_\gamma(x,y)\le C\,d(x,y)^{\gamma d_w-d_h}\,.
\end{equation}
In particular, $G_\gamma$ is symmetric, continuous and finite off the diagonal, and $G_\gamma(x,y)\to\infty$ as $y\to x$.
\end{lem}

\begin{proof}
The proof is similar to  ~\cite[Proposition~2.6]{BaudoinLacaux2022} and~\cite[Lemma~2.3]{BOTW23} and is obtained by integrating the heat kernel bounds.  The continuity off the diagonal follows from the continuity of $p_t$ and dominated convergence, using the upper bound in~\eqref{eq:subGauss}.
\end{proof}

\begin{lem}\label{lem:fractional-laplacian-kernel-is-a-kernel}
Let $\gamma\in(0,\frac{d_h}{2d_w})$. For every $f\in L^1(X,\mu)\cap L^\infty(X,\mu)$, the function $y\mapsto G_\gamma(x,y)f(y)$ is $\mu$-integrable for every $x\in X$, uniformly in $x$, and
\begin{equation}\label{eq:kernel-rep}
(-\Lap)^{-\gamma}f(x)=\int_XG_\gamma(x,y)f(y)\,d\mu(y)\qquad\text{for $\mu$-a.e. }x\in X\,.
\end{equation}
\end{lem}

\begin{proof}
For $f\in L^1\cap L^\infty$ and $0<\gamma<d_h/d_w$, Tonelli's theorem, contractivity and the upper heat kernel bound give uniformly in $x$, by splitting the time integral at $1$:
\begin{equation}\label{eq:G-integrable}
\begin{aligned}
\int_XG_\gamma(x,y)|f(y)|\,d\mu(y)
&=\frac1{\Gamma(\gamma)}\int_0^\infty t^{\gamma-1}P_t|f|(x)\,dt\\
&\le\frac1{\Gamma(\gamma)}\left(
\frac{\|f\|_{L^\infty}}{\gamma}
+\frac{C\|f\|_{L^1}}{d_h/d_w-\gamma}\right).
\end{aligned}
\end{equation}
Fubini's theorem therefore gives, for every $x\in X$,
\begin{equation*}
\int_XG_\gamma(x,y)f(y)\,d\mu(y)=\frac1{\Gamma(\gamma)}\int_0^\infty t^{\gamma-1}P_tf(x)\,dt\,,
\end{equation*}
where the right-hand side converges absolutely. On the other hand, $f\in L^1\cap L^2$, so Lemmas~\ref{lem:Schwarz-properties} and~\ref{lem:fraclap-integral} give $\frac1{\Gamma(\gamma)}\int_0^Rt^{\gamma-1}P_tf\,dt\to(-\Lap)^{-\gamma}f$ in $L^2$ as $R\to\infty$, while it converges pointwise to the right-hand side above. Identifying the limits along a subsequence proves~\eqref{eq:kernel-rep}.
\end{proof}

We will repeatedly use the following covariance identity. For $\alpha>0$, Tonelli's theorem, symmetry and~\eqref{eq:chapman} give, for all $s>0$ and $x,x'\in X$,
\begin{equation}\label{eq:two-point-kernel}
\begin{aligned}
&\int_X\int_Xp_s(x,y)p_s(x',z)G_{2\alpha}(y,z)\,d\mu(y)\,d\mu(z)\\
&\qquad=\frac1{\Gamma(2\alpha)}\int_0^\infty r^{2\alpha-1}p_{r+2s}(x,x')\,dr\\
&\qquad\le\frac{C}{\Gamma(2\alpha)}\int_0^\infty r^{2\alpha-1}(r+2s)^{-\frac{d_h}{d_w}}\,dr
\le Cs^{2\alpha-\frac{d_h}{d_w}}\,.
\end{aligned}
\end{equation}
The last integral is finite because $0<2\alpha<d_h/d_w$. At $\alpha=0$, the corresponding identity is simply
\begin{equation}\label{eq:two-point-kernel-zero}
\int_Xp_s(x,y)p_s(x',y)\,d\mu(y)=p_{2s}(x,x')\le Cs^{-\frac{d_h}{d_w}}\,.
\end{equation}

The following estimate will repeatedly be used.

\begin{lem}\label{lem:sup-kernel-estimates}
Let $\gamma\in[0,\frac{d_h}{2d_w})$. There exists $C>0$ such that for all $t>0$,
\begin{equation}\label{eq:sup-kernel-1}
\sup_{x\in X}\norm{(-\Lap)^{-\gamma}p_t(x,\cdot)}_{L^2}\le Ct^{\gamma-\frac{d_h}{2d_w}}\,.
\end{equation}
\end{lem}

\begin{proof}
For $\gamma=0$, estimate~\eqref{eq:sup-kernel-1} is \eqref{eq:heatkernel-square-bound}. For $\gamma>0$, apply Lemma~\ref{lem:Schwarz-properties} with $f=p_{t/2}(x,\cdot)$ and time $t/2$. Since $\norm{f}_{L^1}=1$ and $P_{t/2}f=p_t(x,\cdot)$, this gives~\eqref{eq:sup-kernel-1}.
\end{proof}

The next identity  is the composition law for negative powers of the Laplacian at the kernel level; see  ~\cite[Remark~3.4]{BaudoinLacaux2022}  for the compact Sierpinski gasket case.

\begin{lem}\label{lem:Kernel-2alpha-correspondence}
Let $\gamma\in(0,\frac{d_h}{2d_w})$. For all $x,y\in X$,
\begin{equation*}
G_{2\gamma}(x,y)=\int_XG_\gamma(x,z)G_\gamma(z,y)\,d\mu(z)
\end{equation*}
(both sides being $+\infty$ when $x=y$).
\end{lem}

\begin{proof}
By Tonelli's theorem and~\eqref{eq:chapman},
\begin{equation*}
\int_XG_\gamma(x,z)G_\gamma(z,y)\,d\mu(z)=\frac1{\Gamma(\gamma)^2}\int_0^\infty\!\!\int_0^\infty(st)^{\gamma-1}p_{s+t}(x,y)\,ds\,dt\,.
\end{equation*}
The change of variables $r=s+t$, $t=ru$ gives
\begin{align*}
\int_0^\infty\!\!\int_0^\infty(st)^{\gamma-1}p_{s+t}(x,y)\,ds\,dt&=\int_0^\infty p_r(x,y)r^{2\gamma-1}dr\int_0^1(1-u)^{\gamma-1}u^{\gamma-1}du\\
&=\Gamma(2\gamma)\Beta(\gamma,\gamma)G_{2\gamma}(x,y)\,,
\end{align*}
and the claim follows from $\Beta(\gamma,\gamma)=\Gamma(\gamma)^2/\Gamma(2\gamma)$.
\end{proof}

The following covariance identity is the analogue of~\cite[Corollary~2.11]{BaudoinChen2023} and~\cite[Proposition~3.1]{BOTW23}.

\begin{lem}\label{lem:sobolev-IP}
Let $\gamma\in(0,\frac{d_h}{2d_w})$ and let $f,g:X\to\R$ be measurable functions such that
\begin{equation}\label{eq:Lambda-cond}
\int_X\int_X|f(x)||f(y)|G_{2\gamma}(x,y)\,d\mu(x)d\mu(y)<\infty\,,
\end{equation}
and likewise for $g$. Then $G_\gamma f(x):=\int_XG_\gamma(x,y)f(y)\,d\mu(y)$ converges absolutely for $\mu$-a.e. $x$, $G_\gamma f\in L^2(X,\mu)$, and
\begin{equation}\label{eq:sobolev-IP}
\int_XG_\gamma f\,G_\gamma g\,d\mu=\int_X\int_Xf(x)g(y)G_{2\gamma}(x,y)\,d\mu(x)d\mu(y)\,,
\end{equation}
the right-hand side being absolutely convergent. In particular the right-hand side of~\eqref{eq:sobolev-IP} is a positive semidefinite bilinear form on the space of functions satisfying~\eqref{eq:Lambda-cond}. Moreover, every $f\in L^1(X,\mu)\cap L^\infty(X,\mu)$ satisfies~\eqref{eq:Lambda-cond}, and for $f,g$ in this space,
\begin{equation}\label{eq:sobolev-IP-S}
\int_X(-\Lap)^{-\gamma}f\,(-\Lap)^{-\gamma}g\,d\mu=\int_X\int_Xf(x)g(y)G_{2\gamma}(x,y)\,d\mu(x)d\mu(y)\,.
\end{equation}
\end{lem}

\begin{proof}
By Lemma~\ref{lem:Kernel-2alpha-correspondence} and Tonelli's theorem,
\begin{equation*}
\int_X\Bigl(\int_XG_\gamma(x,y)|f(y)|\,d\mu(y)\Bigr)^2d\mu(x)=\int_X\int_X|f(y)||f(z)|G_{2\gamma}(y,z)\,d\mu(y)d\mu(z)<\infty\,,
\end{equation*}
which shows that $G_\gamma|f|\in L^2$; in particular $G_\gamma|f|<\infty$ a.e., so $G_\gamma f$ is a.e.\ absolutely convergent and $|G_\gamma f|\le G_\gamma|f|$. The same holds for $g$, and by Cauchy--Schwarz the function $(x,y,z)\mapsto f(y)g(z)G_\gamma(x,y)G_\gamma(x,z)$ is integrable on $X^3$. Fubini's theorem and Lemma~\ref{lem:Kernel-2alpha-correspondence} then give~\eqref{eq:sobolev-IP}, and taking $f=g$ shows the positive semidefiniteness.  If $f\in L^1\cap L^\infty$, then estimate~\eqref{eq:G-integrable},
which was proved for every exponent in $(0,d_h/d_w)$, applies
with $2\gamma$ in place of $\gamma$, since $2\gamma<d_h/d_w$. Thus, $\sup_y\int_X|f(x)|G_{2\gamma}(x,y)\,d\mu(x)<\infty$, so that $\iint|f(x)||f(y)|G_{2\gamma}(x,y)\le C\|f\|_{L^1}<\infty$. Finally, \eqref{eq:sobolev-IP-S} follows from~\eqref{eq:sobolev-IP} and Lemma~\ref{lem:fractional-laplacian-kernel-is-a-kernel}.
\end{proof}

\begin{defn}\label{def:Sobolev-space}
Let $\gamma\in[0,\frac{d_h}{2d_w})$. If $\gamma>0$, let $\Lambda_\gamma$ be the vector space of measurable functions $f:X\to\R$ satisfying~\eqref{eq:Lambda-cond}, endowed with the positive semidefinite bilinear form
\begin{equation*}
\langle f,g\rangle_{\cW^{-\gamma}}=\int_X\int_Xf(x)g(y)G_{2\gamma}(x,y)\,d\mu(x)d\mu(y)\,.
\end{equation*}
The Sobolev space $\cW^{-\gamma}$ is the Hilbert space obtained by completing $\Lambda_\gamma$ with respect to $\langle\cdot,\cdot\rangle_{\cW^{-\gamma}}$. By Lemma~\ref{lem:sobolev-IP}, the map $f\mapsto G_\gamma f$ extends to a linear isometry from $\cW^{-\gamma}$ into $L^2(X,\mu)$, which we also denote by $(-\Lap)^{-\gamma}$. For $f\in L^1(X,\mu)\cap L^\infty(X,\mu)\subset\Lambda_\gamma$, this notation agrees with the spectral definition.

If $\gamma=0$, we set
\begin{equation*}
\Lambda_0=\cW^0=L^2(X,\mu),\qquad
\langle f,g\rangle_{\cW^0}=\int_Xf(x)g(x)\,d\mu(x),\qquad
(-\Lap)^0=I\,.
\end{equation*}
\end{defn}

\begin{lem}\label{lem:semigroup-negative-sobolev}
Let $\gamma\in[0,\frac{d_h}{2d_w})$. For every $f\in\Lambda_\gamma$ and $t>0$, the heat-kernel integral $P_tf$ is $\mu$-a.e.\ absolutely convergent and belongs to $\Lambda_\gamma$. These operators extend uniquely to a strongly continuous contraction semigroup on $\cW^{-\gamma}$ and for every $f\in\cW^{-\gamma}$ and $t>0$,
\begin{equation*}
(-\Lap)^{-\gamma}P_tf=P_t(-\Lap)^{-\gamma}f,
\qquad
\norm{P_tf}_{\cW^{-\gamma}}\le\norm{f}_{\cW^{-\gamma}}\,.
\end{equation*}
\end{lem}

\begin{proof}
For $\gamma=0$, these are the usual $L^2$ properties of the heat semigroup. Suppose that $\gamma>0$ and let $f\in\Lambda_\gamma$ and $g=|f|$. Lemma~\ref{lem:sobolev-IP} gives $G_\gamma g\in L^2$. Tonelli's theorem, \eqref{eq:RieszKernel} and~\eqref{eq:chapman} yield
\begin{equation*}
G_\gamma P_tg=P_tG_\gamma g
\end{equation*}
as an identity of nonnegative functions, initially allowing infinite values. The right-hand side belongs to $L^2$. Since $G_\gamma$ is strictly positive off the diagonal, this also implies that $P_tg$ is finite $\mu$-a.e. By Lemma~\ref{lem:Kernel-2alpha-correspondence} and Tonelli's theorem,
\begin{equation*}
\begin{aligned}
&\int_X\int_X P_tg(x)P_tg(y)G_{2\gamma}(x,y)\,d\mu(x)d\mu(y)\\
&\qquad=\norm{G_\gamma P_tg}_{L^2}^2
=\norm{P_tG_\gamma g}_{L^2}^2
\le\norm{G_\gamma g}_{L^2}^2<\infty\,.
\end{aligned}
\end{equation*}
Thus $P_tg\in\Lambda_\gamma$, and $|P_tf|\le P_tg$ gives $P_tf\in\Lambda_\gamma$. Fubini's theorem now gives $G_\gamma P_tf=P_tG_\gamma f$. The isometry in Definition~\ref{def:Sobolev-space} proves contractivity, so the operators extend to $\cW^{-\gamma}$. The semigroup property and the intertwining identity pass to the completion. Finally, strong continuity on $L^2$ gives, for every $f\in\cW^{-\gamma}$,
\begin{equation*}
\norm{P_tf-f}_{\cW^{-\gamma}}
=\norm{P_t(-\Lap)^{-\gamma}f-(-\Lap)^{-\gamma}f}_{L^2}
\longrightarrow0\qquad(t\downarrow0)\,.\qedhere
\end{equation*}
\end{proof}

\subsection{The noise, the equation and the notion of solution}\label{section:PAM-def-and-existence}

From now on we fix $\alpha$ as in Notation~\ref{not:parameters} and consider $\cW^{-\alpha}$ and $\Lambda_\alpha$ as previously defined. Following the construction in~\cite[Section~3.1]{BOTW23}, we define the noise as an isonormal Gaussian process indexed by $L^2(\R_+;\cW^{-\alpha})$: white in time, with the space covariance given by the norm of $\cW^{-\alpha}$. For $\alpha>0$ this covariance is the pairing against $G_{2\alpha}$, and for $\alpha=0$ it is the $L^2(X,\mu)$ inner product, so that the two cases are treated by the same formalism, the only difference being the inner product~\eqref{eq:H-inner-product} versus~\eqref{eq:H0-inner-product}.

\begin{defn}\label{def:H-alpha}
We let $\cH_\alpha=L^2(\R_+;\cW^{-\alpha})$, i.e.\ the Hilbert space of (classes of) measurable maps $\varphi:\R_+\to\cW^{-\alpha}$ with
\begin{equation*}
\begin{aligned}
\langle\varphi,\psi\rangle_{\cH_\alpha}
&=\int_0^\infty\langle\varphi(t),\psi(t)\rangle_{\cW^{-\alpha}}dt\\
&=\int_0^\infty\langle(-\Lap)^{-\alpha}\varphi(t),(-\Lap)^{-\alpha}\psi(t)\rangle_{L^2(X,\mu)}dt\,.
\end{aligned}
\end{equation*}
If $\alpha>0$ and $\varphi,\psi:\R_+\times X\to\R$ are measurable with $\varphi(t,\cdot),\psi(t,\cdot)\in\Lambda_\alpha$ for a.e.\ $t$, then
\begin{equation}\label{eq:H-inner-product}
\langle\varphi,\psi\rangle_{\cH_\alpha}=\int_0^\infty\int_X\int_X\varphi(t,x)\psi(t,y)G_{2\alpha}(x,y)\,d\mu(x)d\mu(y)dt\,.
\end{equation}
If $\alpha=0$, then $\cH_0=L^2(\R_+\times X,dt\otimes\mu)$ and
\begin{equation}\label{eq:H0-inner-product}
\langle\varphi,\psi\rangle_{\cH_0}=\int_0^\infty\int_X\varphi(t,x)\psi(t,x)\,d\mu(x)dt\,.
\end{equation}
\end{defn}

\begin{defn}\label{def:fractional-noise}
On a complete probability space $(\Omega,\cF,\PP)$, let $\{W_\alpha(\varphi):\varphi\in\cH_\alpha\}$ be an isonormal Gaussian process on $\cH_\alpha$, i.e.\ a centered Gaussian family such that $\varphi\mapsto W_\alpha(\varphi)$ is linear and
\begin{equation*}
\EE\bigl[W_\alpha(\varphi)W_\alpha(\psi)\bigr]=\langle\varphi,\psi\rangle_{\cH_\alpha},\qquad\varphi,\psi\in\cH_\alpha\,.
\end{equation*}
We write $W_\alpha(\varphi)=\int_0^\infty\!\int_X\varphi(s,y)W_\alpha(ds,dy)$. For $t\ge0$ we let $\cF_t$ be the $\sigma$-field generated by $\{W_\alpha(\varphi):\varphi\in\cH_\alpha,\ \varphi(s,\cdot)=0\text{ for }s>t\}$ and the $\PP$-null sets.
\end{defn}

For $\alpha>0$, $\dot W_\alpha$ is a centered Gaussian noise which is white in time and has spatial covariance $G_{2\alpha}(x,y)\simeq d(x,y)^{2\alpha d_w-d_h}$; in the Euclidean case it is the Riesz-type noise with covariance $c|x-y|^{-(d-4\alpha)}$. At $\alpha=0$, which is admissible only when $d_h<d_w$, $W_0$ is space--time white noise: its spatial covariance is the identity on $L^2(X,\mu)$, equivalently the diagonal measure $d\mu(x)\delta_x(dy)$ on $X^2$. In all cases $W_\alpha$ can be viewed as a cylindrical Wiener process on $\cW^{-\alpha}$. 

\vspace{0.4cm}
The stochastic integral $\int_0^t\!\int_X\Phi(s,y)W_\alpha(ds,dy)$ of an $(\cF_t)$-adapted random field $\Phi$ with $\EE\int_0^t\|\Phi(s,\cdot)\|^2_{\cW^{-\alpha}}ds<\infty$ is defined as in Walsh~\cite{Walsh1986} and Dalang~\cite{Dalang1999}, and coincides with the Skorohod integral (divergence operator) of the Malliavin calculus associated with $W_\alpha$~\cite[Section 1.3]{Nualart2006}; see~\cite[Section 2]{HHNT15}, \cite[Section 2]{BalanChen2018} and~\cite[Definitions 3.2--3.3]{BCHOTW24} for this identification in the present type of setting, including the white-noise endpoint. It satisfies the isometry $\EE[(\int_0^t\!\int_X\Phi\,dW_\alpha)^2]=\EE\int_0^t\|\Phi(s,\cdot)\|^2_{\cW^{-\alpha}}ds$ and, for $p\ge2$, the Burkholder--Davis--Gundy inequality
\begin{equation}\label{eq:BDG}
\EE\Bigl[\Bigl|\int_0^t\!\!\int_X\Phi(s,y)W_\alpha(ds,dy)\Bigr|^p\Bigr]\le C_p\,\EE\Bigl[\Bigl(\int_0^t\|\Phi(s,\cdot)\|^2_{\cW^{-\alpha}}ds\Bigr)^{p/2}\Bigr]\,.
\end{equation}
\vspace{0.4cm}

We can now state precisely what we mean by a solution of the parabolic Anderson model
\begin{equation}\label{eq:pam-D}
\partial_{t}u(t,x)=\Lap u(t,x)+\beta\,u(t,x)\,\dot{W}_\alpha(t,x),\qquad u(0,\cdot)=u_0\in L^\infty(X,\mu)\,,
\end{equation}
which we interpret in the mild (It\^o--Skorohod) sense, as in~\cite[Definition~4.5]{BOTW23}.

\begin{defn}\label{def:mild-PAM-solution}
Let $u_0\in L^\infty(X,\mu)$ and $\beta>0$. A random field $u=\{u(t,x):(t,x)\in\R_+\times X\}$ is a (mild) solution of~\eqref{eq:pam-D} if:
\begin{enumerate}\itemsep=0pt
\item $u$ is $(\cF_t)$-adapted and jointly measurable with respect to $\mathcal B(\R_+\times X)\otimes\cF$;
\item for all $(t,x)\in\R_+\times X$, the random field $(s,y)\mapsto p_{t-s}(x,y)u(s,y)\1_{[0,t]}(s)$ is stochastically integrable, i.e.
\begin{equation*}
I(t,x):=\int_0^t\!\!\int_Xp_{t-s}(x,y)u(s,y)\,W_\alpha(ds,dy)
\end{equation*}
is well defined and $\EE[I(t,x)^2]<\infty$;
\item for all $(t,x)\in\R_+\times X$,
\begin{equation}\label{eq:mild}
u(t,x)=P_tu_0(x)+\beta I(t,x)\qquad\PP\text{-a.s.}\,.
\end{equation}
\end{enumerate}
\end{defn}

\subsection{Wiener chaos expansion}\label{section:Chaos_expansion}

Iterating the mild equation~\eqref{eq:mild} expresses the solution as a series of multiple Wiener--It\^o integrals whose kernels are products of heat kernels. This expansion is both the construction of the solution and the tool for obtaining both the upper and lower moment bounds. Let $I_k$ denote the $k$-th multiple Wiener-It\^o integral with respect to $W_\alpha$, $k\ge1$~\cite[Section 1.1]{Nualart2006}. Recall that $I_k:\cH_\alpha^{\otimes k}\to L^2(\Omega)$ is linear, that $I_k(\psi)=I_k(\tilde\psi)$ where $\tilde\psi$ is the symmetrization of $\psi$, that $\EE[I_k(\psi)]=0$, that the images of the $I_k$'s are mutually orthogonal in $L^2(\Omega)$ and orthogonal to the constants, and that
\begin{equation}\label{eq:Ik-isometry}
\EE\bigl[I_k(\psi)^2\bigr]=k!\,\|\tilde\psi\|^2_{\cH_\alpha^{\otimes k}}\le k!\,\|\psi\|^2_{\cH_\alpha^{\otimes k}},\qquad\psi\in\cH_\alpha^{\otimes k}\,.
\end{equation}
Every $F\in L^2(\Omega,\sigma(W_\alpha),\PP)$ admits the chaos expansion $$F=\EE[F]+\sum_{k\ge1}I_k(\psi_k)$$ with uniquely determined symmetric kernels $\psi_k\in\cH_\alpha^{\otimes k}$. For $k\ge1$, $(t,x)\in\R_+\times X$, $s=(s_1,\dots,s_k)\in\R_+^k$ and $y=(y_1,\dots,y_k)\in X^k$, define
\begin{equation}\label{eq:fk}
\begin{split}
f_k(s,y;t,x)=\frac{\beta^k}{k!}\,p_{t-s_{\sigma(k)}}(x,y_{\sigma(k)})\,p_{s_{\sigma(k)}-s_{\sigma(k-1)}}(y_{\sigma(k)},y_{\sigma(k-1)})\cdots p_{s_{\sigma(2)}-s_{\sigma(1)}}(y_{\sigma(2)},y_{\sigma(1)})\\
\times P_{s_{\sigma(1)}}u_0(y_{\sigma(1)})\,\1_{\{s_{\sigma(k)}<t\}}\,,
\end{split}
\end{equation}
where $\sigma$ is the permutation of $\{1,\dots,k\}$ such that $s_{\sigma(1)}<\dots<s_{\sigma(k)}$ (the set where two $s_i$ coincide being Lebesgue-null). By construction $f_k(\cdot;t,x)$ is symmetric in the $k$ variables $(s_i,y_i)$, and it vanishes unless all $s_i$ belong to $(0,t)$. The following standard facts are obtained by iterating the mild equation~\eqref{eq:mild} and using the relation between the Skorohod integral and multiple integrals~\cite[Section 1.3.1]{Nualart2006}; see~\cite[Section 4]{HuNualart2009}, \cite[Section~3.1]{HHNT15}, \cite[Section 2]{BalanChen2018}, \cite[Proposition 3.7]{BCHOTW24}, and~\cite[Section 4.1.1 and Proposition 6.2.2]{TindelAarhus2026} for detailed proofs in closely related settings, including the white-noise endpoint.

\begin{prop}\label{prop:chaos-exp-summability-implies-solution}
Let $u_0\in L^\infty(X,\mu)$ and $\beta>0$.
\begin{enumerate}
\item If $u$ is a solution of~\eqref{eq:pam-D} in the sense of Definition~\ref{def:mild-PAM-solution}, then for every $(t,x)\in\R_+\times X$ one has $f_k(\cdot;t,x)\in\cH_\alpha^{\otimes k}$ for all $k\ge1$ and
\begin{equation}\label{eq:chaos-expansion}
u(t,x)=P_tu_0(x)+\sum_{k=1}^\infty I_k\bigl(f_k(\cdot;t,x)\bigr)\quad\text{in }L^2(\Omega)\,,
\end{equation}
so that, by orthogonality and~\eqref{eq:Ik-isometry},
\begin{equation}\label{eq:second-moment-chaos}
\EE[u(t,x)^2]=(P_tu_0(x))^2+\sum_{k=1}^\infty k!\,\|f_k(\cdot;t,x)\|^2_{\cH_\alpha^{\otimes k}}\,.
\end{equation}
\item Conversely, if $f_k(\cdot;t,x)\in\cH_\alpha^{\otimes k}$ and
\begin{equation}\label{eq:summability}
\sum_{k=1}^\infty k!\,\|f_k(\cdot;t,x)\|^2_{\cH_\alpha^{\otimes k}}<\infty\qquad\text{for every }(t,x)\in\R_+\times X\,,
\end{equation}
then the random field defined by~\eqref{eq:chaos-expansion} is a solution of~\eqref{eq:pam-D}, and it is the unique solution.
\end{enumerate}
\end{prop}

We now estimate the chaos kernels by a recursion based on the covariance pairing of two heat kernels. Throughout, $\alpha$ satisfies~\eqref{eq:alpha-range}, $\beta>0$ and $u_0\in L^\infty(X,\mu)$. We write $\cH=\cH_\alpha$ and $f_k(\cdot;t,x)$ for the kernels~\eqref{eq:fk}.

\begin{lem}\label{lem:expansion-l2-estimates}
There exists a constant $C_\alpha>0$, depending only on the structural constants and on $\alpha$, such that for all $k\ge1$, $\rho>0$, $t>0$ and $x\in X$, $f_k(\cdot;t,x)\in\cH^{\otimes k}$ and
\begin{equation}\label{eq:key-estimate}
\norm{f_k(\cdot;t,x)}^2_{\cH^{\otimes k}}\le\frac{\beta^{2k}}{k!}\,\norm{u_0}_{L^\infty}^2\,C_\alpha^k\,\rho^{-k(1+2\alpha-\frac{d_h}{d_w})}\,e^{\rho t}\,.
\end{equation}
\end{lem}

\begin{proof}
We use the exponentially weighted recursion of~\cite[Section~3.2]{BCO24} and~\cite[Sections~3--4]{FoondunKhoshnevisan2009}. Starting from $v_0(t,x)=P_tu_0(x)$, construct the chaos terms successively by
\begin{equation*}
v_{k+1}(t,x)=\beta\int_0^t\!\int_Xp_{t-s}(x,z)v_k(s,z)\,W_\alpha(ds,dz)\,.
\end{equation*}
The estimates below justify this construction at each step. Set
\begin{equation}\label{eq:Mk-def}
M_k(t):=\sup_{x\in X}\EE[|v_k(t,x)|^2],\qquad k\ge0\,.
\end{equation}
In particular, $M_0(t)\le\|u_0\|_{L^\infty}^2$.

Suppose that $v_k$ has been constructed. For $\alpha>0$, the stochastic-integral isometry gives
\begin{align*}
\EE[|v_{k+1}(t,x)|^2]
&=\beta^2\int_0^t\int_{X^2}p_{t-s}(x,z)p_{t-s}(x,z')G_{2\alpha}(z,z')\\
&\qquad\times\EE[v_k(s,z)v_k(s,z')]\,d\mu(z)d\mu(z')\,ds\,.
\end{align*}
By Cauchy--Schwarz,
\begin{equation*}
|\EE[v_k(s,z)v_k(s,z')]|\le
\EE[|v_k(s,z)|^2]^{1/2}\EE[|v_k(s,z')|^2]^{1/2}\le M_k(s)\,.
\end{equation*}
Since $G_{2\alpha}\ge0$, estimate~\eqref{eq:two-point-kernel} yields
\begin{equation}\label{eq:Mk-recursion}
M_{k+1}(t)\le C\beta^2\int_0^t(t-s)^{2\alpha-d_h/d_w}M_k(s)\,ds\,.
\end{equation}
For $\alpha=0$, the same conclusion follows from
\begin{align*}
\EE[|v_{k+1}(t,x)|^2]
&=\beta^2\int_0^t\int_Xp_{t-s}(x,z)^2\EE[|v_k(s,z)|^2]\,d\mu(z)\,ds\\
&\le\beta^2\int_0^t p_{2(t-s)}(x,x)M_k(s)\,ds
\le C\beta^2\int_0^t(t-s)^{-d_h/d_w}M_k(s)\,ds\,
\end{align*}
see also~\cite[Lemma~3.9]{BCHOTW24}. These calculations also bound the expected squared norm of the stochastic integrand. Thus finiteness of the right-hand side of~\eqref{eq:Mk-recursion} suffices to define $v_{k+1}$ and justify the isometry.

Fix $\rho>0$ and choose $C_\alpha=C\Gamma(1+2\alpha-d_h/d_w)$, where $C$ is the constant in~\eqref{eq:Mk-recursion}. We prove by induction that
\begin{equation}\label{eq:Mk-bound}
M_k(t)\le\|u_0\|_{L^\infty}^2e^{\rho t}
\left(C_\alpha\beta^2\rho^{d_h/d_w-1-2\alpha}\right)^k\,.
\end{equation}
The case $k=0$ follows from the $L^\infty$ contraction of $P_t$. If the bound holds for $k$, then~\eqref{eq:Mk-recursion} gives
\begin{align*}
M_{k+1}(t)
&\le\|u_0\|_{L^\infty}^2e^{\rho t}
\left(C_\alpha\beta^2\rho^{d_h/d_w-1-2\alpha}\right)^k
C\beta^2\int_0^t e^{-\rho r}r^{2\alpha-d_h/d_w}\,dr\\
&\le\|u_0\|_{L^\infty}^2e^{\rho t}
\left(C_\alpha\beta^2\rho^{d_h/d_w-1-2\alpha}\right)^{k+1}\,.
\end{align*}
Here we used
\begin{equation*}
\int_0^\infty e^{-\rho r}r^{2\alpha-d_h/d_w}\,dr
=\Gamma\left(1+2\alpha-\frac{d_h}{d_w}\right)\rho^{d_h/d_w-1-2\alpha}<\infty\,,
\end{equation*}
which is finite by~\eqref{eq:alpha-range}. This proves both the recursive construction and~\eqref{eq:Mk-bound}.

Iterating the stochastic integrals and symmetrizing their ordered kernels gives $v_k(t,x)=I_k(f_k(\cdot;t,x))$, with $f_k$ as in~\eqref{eq:fk}; see~\cite[Section~1.3.1]{Nualart2006}. In particular, $f_k(\cdot;t,x)\in\cH^{\otimes k}$, and the isometry gives
\begin{equation}\label{eq:fk-vs-Mk}
k!\,\norm{f_k(\cdot;t,x)}_{\cH^{\otimes k}}^2
=\EE[|v_k(t,x)|^2]\le M_k(t)\,.
\end{equation}
Combining~\eqref{eq:fk-vs-Mk} with~\eqref{eq:Mk-bound} proves~\eqref{eq:key-estimate}.
\end{proof}

\subsection{Existence, uniqueness and finite-time moment bounds}\label{sec:finite-time}

The geometric bound~\eqref{eq:key-estimate} makes the chaos series summable once $\rho$ is large enough for $C_\alpha\beta^2\rho^{d_h/d_w-1-2\alpha}$ to be less than one.

\begin{thm}\label{thm:exist-unique-wholespace}
 Let $\alpha$ satisfy~\eqref{eq:alpha-range}, $\beta>0$ and $u_0\in L^\infty(X,\mu)$. Then the equation
\begin{equation*}
\partial_{t}u(t,x)=\Lap u(t,x)+\beta\,u(t,x)\,\dot{W}_\alpha(t,x),\qquad u(0,\cdot)=u_0\in L^\infty(X,\mu)\,,
\end{equation*}
has a unique solution $u$ in the sense of Definition~\ref{def:mild-PAM-solution}. It is given by the chaos expansion~\eqref{eq:chaos-expansion}, and for every $\rho>\rho_c:=(C_\alpha\beta^2)^{\frac{d_w}{(1+2\alpha)d_w-d_h}}$, with $C_\alpha$ as in Lemma~\ref{lem:expansion-l2-estimates}, and all $(t,x)\in\R_+\times X$,
\begin{equation}\label{eq:chaos-sum-bound}
\mathbb{E}(u(t,x)^2)\le\norm{u_0}^2_{L^\infty}\,e^{\rho t}\sum_{k=0}^\infty\bigl(C_\alpha\beta^2\rho^{\frac{d_h}{d_w}-1-2\alpha}\bigr)^k<\infty\,.
\end{equation}
Moreover, if $u_0\ge0$ $\mu$-a.e., then $u(t,x)\ge0$ $\PP$-a.s. for every $(t,x)$. If, in addition, $u_0\ne0$ in $L^\infty(X,\mu)$, then $u(t,x)>0$ $\PP$-a.s. for every $(t,x)\in\R_+\times X$.
\end{thm}

\begin{proof}
We use Proposition  \ref{prop:chaos-exp-summability-implies-solution}. By Lemma~\ref{lem:expansion-l2-estimates}, $k!\|f_k(\cdot;t,x)\|^2_{\cH^{\otimes k}}\le\|u_0\|^2_{L^\infty}e^{\rho t}\delta_\rho^k$ with $\delta_\rho:=C_\alpha\beta^2\rho^{\frac{d_h}{d_w}-1-2\alpha}$, and $\delta_\rho<1$ if and only if $\rho>\rho_c$, since $\rho\mapsto\rho^{\frac{d_h}{d_w}-1-2\alpha}$ is strictly decreasing. Fix such a $\rho$. Orthogonality gives~\eqref{eq:chaos-sum-bound}.

Suppose now that $u_0\ge0$ $\mu$-a.e. We sketch the usual Feynman-Kac regularization argument which is used to prove non-negativity of $u(t,x)$; see~\cite[Section~3.2]{HHNT15}. For $\eps>0$, we smooth the noise in space by setting $W_{\alpha,\eps}(\varphi)=W_\alpha(P_\eps\varphi)$, where $P_\eps$ acts on the spatial variable. The regularized solution of the corresponding PAM has the Wick Feynman--Kac representation
\begin{equation*}
u_\eps(t,x)=\EE_x\Bigl[u_0(B_t)\exp\Bigl(\beta W_\alpha(g_{t,B}^\eps)-\frac{\beta^2}{2}\|g_{t,B}^\eps\|_{\cH}^2\Bigr)\Bigr]\ge0\,,
\end{equation*}
where $g_{t,B}^\eps(s,y)=\1_{(0,t)}(s)p_\eps(B_{t-s},y)$ and $\EE_x$ averages only over the diffusion, independently of the noise. The bound~\eqref{eq:sup-kernel-1} gives $\|g_{t,B}^\eps\|_{\cH}^2\le Ct\eps^{2\alpha-d_h/d_w}$, so the expression is well defined. Expanding the Wick exponential and using the Markov property identifies the chaos kernels as $f_{k,\eps}=P_\eps^{\otimes k}f_k$, with $P_\eps$ acting on each spatial variable. By Lemma~\ref{lem:semigroup-negative-sobolev}, $P_\eps$ is a strongly continuous contraction on $\cW^{-\alpha}$, so these kernels converge in $\cH^{\otimes k}$ to $f_k$ and their norms are bounded by $\|f_k\|_{\cH^{\otimes k}}$. The summability in~\eqref{eq:chaos-sum-bound} therefore gives $u_\eps(t,x)\to u(t,x)$ in $L^2(\Omega)$, proving nonnegativity for bounded measurable initial conditions, including when $\alpha=0$.

For strict positivity, fix $t>0$ and $x\in X$, and write $u_\beta$ for the solution with coupling constant $\beta$. Let $(\mathcal Q_s)_{s\ge0}$ be the Ornstein-Uhlenbeck semigroup on $L^2(\Omega,\sigma(W_\alpha),\PP)$, characterized by $\mathcal Q_sI_k(f)=e^{-ks}I_k(f)$~\cite[Section~1.4.3]{Nualart2006}. This symmetric Markov semigroup is  positivity improving by Mehler's formula. Since the $k$-th chaos kernel in~\eqref{eq:fk} is proportional to $\beta^k$, the $L^2$-convergent chaos expansions give
\begin{equation*}
\mathcal Q_s\bigl(u_{\beta e^s}(t,x)\bigr)=u_\beta(t,x),\qquad s>0\,.
\end{equation*}
If $u_0\ge0$ and $u_0\ne0$ in $L^\infty(X,\mu)$, then $u_{\beta e^s}(t,x)\ge0$ almost surely, and the chaos expansion gives $\EE[u_{\beta e^s}(t,x)]=P_tu_0(x)>0$ by~\eqref{eq:subGauss}. Positivity improvement therefore yields $u_\beta(t,x)>0$ almost surely.
\end{proof}

\begin{remark}
For comparison principles for  general stochastic reaction--diffusion equations on metric measure spaces, see~\cite[Theorem~1.4]{FanSunYang2026}.
\end{remark}

The same argument yields $L^p$-bounds for every real $p\ge2$: hypercontractivity bounds the $L^p(\Omega)$ norm of the $k$-th chaos by $(p-1)^{k/2}$ times its $L^2(\Omega)$ norm, which replaces $\beta^2$ by $(p-1)\beta^2$ in the geometric ratio.

\begin{prop}\label{prop:Lp-estimates}
Under the assumptions of Theorem~\ref{thm:exist-unique-wholespace}, let $p\ge2$ be a real number and set $\rho_c(p):=(C_\alpha(p-1)\beta^2)^{\frac{d_w}{(1+2\alpha)d_w-d_h}}$, where $C_\alpha$ is the constant in Lemma~\ref{lem:expansion-l2-estimates}. Then the chaos series~\eqref{eq:chaos-expansion} converges in $L^p(\Omega)$, and for every $\rho>\rho_c(p)$ there is $C_{\rho,p}<\infty$ such that
\begin{equation}\label{eq:Lp-bound}
\sup_{x\in X}\norm{u(t,x)}_{L^p(\Omega)}\le C_{\rho,p}\,e^{\rho t/2}\,\norm{u_0}_{L^\infty}\qquad\text{for all }t>0\,.
\end{equation}
In particular, for every $T>0$,
\begin{equation}\label{eq:finite-time-moments}
\sup_{(t,x)\in(0,T]\times X}\EE\bigl[|u(t,x)|^p\bigr]<\infty\,.
\end{equation}
Moreover, for every $x\in X$,
\begin{equation}\label{eq:Lp-limsup}
\limsup_{t\to\infty}\frac{\log\EE[|u(t,x)|^p]}{t}\le\frac p2\,\rho_c(p)=\frac{p}{2}\,C_\alpha^{\frac{d_w}{(1+2\alpha)d_w-d_h}}\,\bigl(\beta^2(p-1)\bigr)^{\frac{d_w}{(1+2\alpha)d_w-d_h}}\,.
\end{equation}
\end{prop}

\begin{proof}
As in~\cite[Section~3.3]{BCO24}, we apply the hypercontractivity of the Ornstein--Uhlenbeck semigroup~\cite[Section 1.4.3]{Nualart2006}: for every $\psi\in\cH^{\otimes k}$ one has $\|I_k(\psi)\|_{L^p(\Omega)}\le(p-1)^{k/2}\|I_k(\psi)\|_{L^2(\Omega)}$. Hence, by the triangle inequality, \eqref{eq:Ik-isometry}, \eqref{eq:semigroup-infinity-bound} and Lemma~\ref{lem:expansion-l2-estimates},
\begin{align*}
\norm{u(t,x)}_{L^p(\Omega)}&\le|P_tu_0(x)|+\sum_{k\ge1}(p-1)^{k/2}\sqrt{k!}\,\norm{f_k(\cdot;t,x)}_{\cH^{\otimes k}}\\
&\le\norm{u_0}_{L^\infty}\Bigl(1+e^{\rho t/2}\sum_{k\ge1}\bigl(C_\alpha(p-1)\beta^2\rho^{\frac{d_h}{d_w}-1-2\alpha}\bigr)^{k/2}\Bigr)\,,
\end{align*}
and the series converges if and only if $\rho>\rho_c(p)$. This proves the $L^p$-convergence of the chaos series and~\eqref{eq:Lp-bound} with $C_{\rho,p}=\sum_{k\ge0}(C_\alpha(p-1)\beta^2\rho^{\frac{d_h}{d_w}-1-2\alpha})^{k/2}$. Raising~\eqref{eq:Lp-bound} to the power $p$, taking logarithms, dividing by $t$, letting $t\to\infty$ and then $\rho\downarrow\rho_c(p)$ gives~\eqref{eq:Lp-limsup}.
\end{proof}

\subsection{H\"older regularity}\label{sec:holder}

We now prove that the solution is H\"older continuous and relate the H\"older regularity of the solution $u(t,x)$ to the geometry of $X$. The argument combines estimates for the space and time increments of the heat kernel with the Burkholder--Davis--Gundy inequality, adapting the Euclidean method of~\cite{ChenHuang2019}\footnote{The estimate~(3.2) in~\cite[Lemma~3.1]{ChenHuang2019} is not correct: in the notation of that paper, taking $x=0$, $t=1$ and $t'\to\infty$ gives a contradiction whenever $d>\alpha$. Lemma~\ref{lem:weighted-increments} below keeps heat kernels at both times and avoids this problem.}.

We formulate the estimates in terms of an exponent $\kappa\in(0,1]$ for which the heat semigroup satisfies the weak Bakry--\'Emery condition $\mathrm{wBE}(\kappa)$ of~\cite{Besov3}: there exists $C>0$ such that, for all $t>0$, $g\in L^\infty(X,\mu)$ and $x,y\in X$,
\begin{equation}\label{eq:wBE}
|P_tg(x)-P_tg(y)|\le C\frac{d(x,y)^\kappa}{t^{\kappa/d_w}}\norm{g}_{L^\infty}\,.
\end{equation}
Under Assumption~\ref{hyp:dirichlet-form}, Ahlfors regularity and the two-sided sub-Gaussian estimates imply a parabolic Harnack inequality and interior H\"older estimates for caloric functions~\cite[Theorem~3.1 and Corollary~4.2]{BarlowGrigoryanKumagai2012}. Applying these estimates to $(s,z)\mapsto P_sg(z)$ on cylinders of radius comparable to $t^{1/d_w}$, and using the $L^\infty$ contraction of $P_s$, yields~\eqref{eq:wBE} for some exponent $\kappa_0\in(0,1)$ depending only on the structural constants. This structural exponent however need not be optimal. 
The weak Bakry-\'Emery estimate ~\eqref{eq:wBE} is therefore always satisfied for some exponent in our setting. The additional information  is the specified value of $\kappa$, which may reflect finer geometric properties of $X$; see~\cite{Besov3}.  For instance, Condition~\eqref{eq:wBE} holds with $\kappa=1$ on complete Riemannian manifolds with nonnegative Ricci curvature. It also holds on nested fractals and Sierpi\'nski carpets; in particular, on the Sierpi\'nski gasket the optimal $\kappa=d_w-d_h$~\cite[Theorem~3.7 and Section~5.1]{Besov3}. On the Sierpi\'nski carpet the optimal $\kappa$ is not known.

Throughout this subsection, $\alpha$ satisfies~\eqref{eq:alpha-range}, and $u$ is the solution of~\eqref{eq:pam-D} given by Theorem~\ref{thm:exist-unique-wholespace}.

\subsubsection{H\"older estimates for the heat kernel}

\begin{lem}\label{lem:HK-holder-time-derivative}
\begin{enumerate}
\item For all $t>0$ and $x,y,z\in X$,
\begin{equation}\label{eq:HK-space-inc}
|p_t(x,z)-p_t(y,z)|\le C\frac{d(x,y)^\kappa}{t^{(\kappa+d_h)/d_w}}\,.
\end{equation}
\item For all $x,z\in X$, the map $t\mapsto p_t(x,z)$ is continuously differentiable on $(0,\infty)$ and
\begin{equation}\label{eq:HK-time-derivative}
|\partial_tp_t(x,z)|\le Ct^{-1-\frac{d_h}{d_w}}\,.
\end{equation}
\end{enumerate}
\end{lem}

\begin{proof}
(i) is \cite[Lemma~3.4]{Besov3}. For (ii), by \cite[Corollary~5.7 and
Remark~5.8]{GrigoryanTelcs2012}, $t\mapsto p_t(x,z)$ is differentiable on
$(0,\infty)$ for all $x,z\in X$, with
$|\partial_tp_t(x,z)|\le \frac2t\sqrt{p_{t/2}(x,x)\,p_{t/2}(z,z)}$; the upper
bound in~\eqref{eq:subGauss} then gives~\eqref{eq:HK-time-derivative}. Continuity of the derivative follows from
$\partial_tp_t(x,z)=\langle\Lap p_{t/2}(x,\cdot),p_{t/2}(z,\cdot)\rangle_{L^2}$
and the analyticity of the semigroup.
\end{proof}

The next lemma combines the uniform H\"older estimates with the off-diagonal decay in~\eqref{eq:subGauss}, retaining heat kernels at both times in the time-increment bound. 
\begin{lem}\label{lem:weighted-increments}
For every $\vartheta\in(0,1)$ there exist constants $C>0$ and $c_\vartheta\ge1$ such that:
\begin{enumerate}
\item for all $t>0$ and $x,y,z\in X$,
\begin{equation}\label{eq:weighted-space}
|p_t(x,z)-p_t(y,z)|\le C\Bigl(\frac{d(x,y)}{t^{1/d_w}}\Bigr)^{\kappa\vartheta}\bigl[p_{c_\vartheta t}(x,z)+p_{c_\vartheta t}(y,z)\bigr]\,;
\end{equation}
\item for all $0<s<t\le2s$ and $y,z\in X$,
\begin{equation}\label{eq:weighted-time}
|p_t(y,z)-p_s(y,z)|\le C\Bigl(\frac{t-s}{s}\Bigr)^{\vartheta}\bigl[p_{c_\vartheta t}(y,z)+p_{c_\vartheta s}(y,z)\bigr]\,.
\end{equation}
\end{enumerate}
\end{lem}

\begin{proof}
For every $c'>0$, the lower bound in~\eqref{eq:subGauss} gives constants $C>0$ and $c''\ge1$ such that, for all $t>0$ and $x,z\in X$,
\begin{equation}\label{eq:exp-to-kernel}
t^{-\frac{d_h}{d_w}}\exp\Bigl(-c'\Bigl(\frac{d(x,z)^{d_w}}{t}\Bigr)^{\frac1{d_w-1}}\Bigr)\le C p_{c''t}(x,z)\,.
\end{equation}
Indeed, a sufficiently large time dilation makes the exponential decay in the heat kernel lower bound slower than that on the left. Fix $\vartheta\in(0,1)$. Combining the upper bound in~\eqref{eq:subGauss} with~\eqref{eq:exp-to-kernel}, and choosing $c_\vartheta\ge1$ sufficiently large, gives, for all $r>0$ and $a,z\in X$,
\begin{equation*}
p_r(a,z)^{1-\vartheta}\le C r^{\vartheta d_h/d_w}p_{c_\vartheta r}(a,z)\,.
\end{equation*}
We use this comparison for each kernel separately, together with $(A+B)^{1-\vartheta}\le A^{1-\vartheta}+B^{1-\vartheta}$ for $A,B\ge0$.

(i) Since $|p_t(x,z)-p_t(y,z)|\le p_t(x,z)+p_t(y,z)$, interpolation with~\eqref{eq:HK-space-inc} yields
\begin{align*}
|p_t(x,z)-p_t(y,z)|
&\le |p_t(x,z)-p_t(y,z)|^\vartheta\bigl[p_t(x,z)^{1-\vartheta}+p_t(y,z)^{1-\vartheta}\bigr] \\
&\le C\Bigl(\frac{d(x,y)}{t^{1/d_w}}\Bigr)^{\kappa\vartheta}t^{-\vartheta d_h/d_w}\bigl[p_t(x,z)^{1-\vartheta}+p_t(y,z)^{1-\vartheta}\bigr]\,.
\end{align*}
Applying the preceding comparison with $r=t$ proves~\eqref{eq:weighted-space}.

(ii) Let $0<s<t\le2s$. By~\eqref{eq:HK-time-derivative},
\begin{equation*}
|p_t(y,z)-p_s(y,z)|\le\int_s^t|\partial_rp_r(y,z)|\,dr\le C\frac{t-s}{s}s^{-d_h/d_w}\,.
\end{equation*}
Interpolating this bound with $|p_t(y,z)-p_s(y,z)|\le p_t(y,z)+p_s(y,z)$ and applying the comparison at times $t$ and $s$, we obtain
\begin{align*}
|p_t(y,z)-p_s(y,z)|
&\le C\Bigl(\frac{t-s}{s}\Bigr)^\vartheta s^{-\vartheta d_h/d_w}\bigl[p_t(y,z)^{1-\vartheta}+p_s(y,z)^{1-\vartheta}\bigr] \\
&\le C\Bigl(\frac{t-s}{s}\Bigr)^\vartheta\Bigl[\Bigl(\frac ts\Bigr)^{\vartheta d_h/d_w}p_{c_\vartheta t}(y,z)+p_{c_\vartheta s}(y,z)\Bigr]\,.
\end{align*}
Since $t/s\le2$, the factor $(t/s)^{\vartheta d_h/d_w}$ is bounded, and~\eqref{eq:weighted-time} follows.
\end{proof}

We now translate Lemma~\ref{lem:weighted-increments} into $\cW^{-\alpha}$-bounds for the \emph{absolute} differences of the heat kernel.

\begin{cor}\label{cor:W-increments}
For every $\vartheta\in(0,1)$ there exists $C>0$ such that:
\begin{enumerate}
\item for all $t>0$ and $x,y\in X$, the function $|p_t(x,\cdot)-p_t(y,\cdot)|$ belongs to $\Lambda_\alpha$ and
\begin{equation}\label{eq:W-space}
\bigl\||p_t(x,\cdot)-p_t(y,\cdot)|\bigr\|^2_{\cW^{-\alpha}}\le C\,d(x,y)^{2\kappa\vartheta}\,t^{2\alpha-\frac{d_h}{d_w}-\frac{2\kappa\vartheta}{d_w}}\,;
\end{equation}
\item for all $0<s<t$ and $y\in X$, the function $|p_t(y,\cdot)-p_s(y,\cdot)|$ belongs to $\Lambda_\alpha$ and
\begin{equation}\label{eq:W-time}
\bigl\||p_t(y,\cdot)-p_s(y,\cdot)|\bigr\|^2_{\cW^{-\alpha}}\le C\,(t-s)^{2\vartheta}\,s^{2\alpha-\frac{d_h}{d_w}-2\vartheta}\,.
\end{equation}
\end{enumerate}
\end{cor}

\begin{proof}
Since $p_t(x,\cdot)\in L^1(X,\mu)\cap L^\infty(X,\mu)$, Definition~\ref{def:Sobolev-space} and~\eqref{eq:sup-kernel-1} give
\begin{equation*}
\|p_t(x,\cdot)\|^2_{\cW^{-\alpha}}=\|(-\Lap)^{-\alpha}p_t(x,\cdot)\|^2_{L^2}\le Ct^{2\alpha-\frac{d_h}{d_w}},\qquad t>0,\ x\in X\,.
\end{equation*}
Also notice that if $0\le f\le F$ with $F\in\Lambda_\alpha$, then $f\in\Lambda_\alpha$ and $\|f\|_{\cW^{-\alpha}}\le\|F\|_{\cW^{-\alpha}}$. For $\alpha>0$ this follows from $G_{2\alpha}\ge0$ and Definition~\ref{def:Sobolev-space}; for $\alpha=0$ it is trivial.

\vspace{0.4cm}

(i) By~\eqref{eq:weighted-space}, $0\le|p_t(x,\cdot)-p_t(y,\cdot)|\le F:=C(d(x,y)t^{-1/d_w})^{\kappa\vartheta}[p_{c_\vartheta t}(x,\cdot)+p_{c_\vartheta t}(y,\cdot)]$, and $F\in L^1(X,\mu)\cap L^\infty(X,\mu)\subset\Lambda_\alpha$. Hence, by the triangle inequality in $\cW^{-\alpha}$,
\begin{align*}
\bigl\||p_t(x,\cdot)-p_t(y,\cdot)|\bigr\|_{\cW^{-\alpha}}
&\le C\Bigl(\frac{d(x,y)}{t^{1/d_w}}\Bigr)^{\kappa\vartheta}
\bigl(\|p_{c_\vartheta t}(x,\cdot)\|_{\cW^{-\alpha}}+\|p_{c_\vartheta t}(y,\cdot)\|_{\cW^{-\alpha}}\bigr)\\
&\le C\Bigl(\frac{d(x,y)}{t^{1/d_w}}\Bigr)^{\kappa\vartheta}t^{\alpha-\frac{d_h}{2d_w}}\,,
\end{align*}
which is~\eqref{eq:W-space}.

(ii) If $t\le2s$, \eqref{eq:weighted-time} and the same argument give
\begin{align*}
\bigl\||p_t(y,\cdot)-p_s(y,\cdot)|\bigr\|_{\cW^{-\alpha}}&\le C\Bigl(\frac{t-s}{s}\Bigr)^\vartheta\bigl(\|p_{c_\vartheta t}(y,\cdot)\|_{\cW^{-\alpha}}+\|p_{c_\vartheta s}(y,\cdot)\|_{\cW^{-\alpha}}\bigr)\\
&\le C(t-s)^\vartheta s^{-\vartheta}s^{\alpha-\frac{d_h}{2d_w}}\,,
\end{align*}
where we used $t\ge s$ and $2\alpha<d_h/d_w$. If $t>2s$, then $(t-s)/s>1$, and simply
\begin{equation*}
\bigl\||p_t(y,\cdot)-p_s(y,\cdot)|\bigr\|_{\cW^{-\alpha}}\le\|p_t(y,\cdot)\|_{\cW^{-\alpha}}+\|p_s(y,\cdot)\|_{\cW^{-\alpha}}\le Cs^{\alpha-\frac{d_h}{2d_w}}\le C\Bigl(\frac{t-s}{s}\Bigr)^\vartheta s^{\alpha-\frac{d_h}{2d_w}}\,.\qedhere
\end{equation*}
\end{proof}

\subsubsection{H\"older regularity of the solution}

We are now in position to prove the required H\"older continuity statement which is based on the metric-space Kolmogorov-Chentsov theorem~\cite[Theorem~1.1]{KraetschmerUrusov2023}.

\begin{thm}\label{thm:Holder}
Assume~\eqref{eq:wBE} and let $u$ be the solution of~\eqref{eq:pam-D} with $u_0\in L^\infty(X,\mu)$ and $\alpha$ satisfying~\eqref{eq:alpha-range}. Let $p\ge2$, $0<T<T'<\infty$,
\begin{equation*}
\vartheta_t\in\Bigl(0,\frac{(1+2\alpha)d_w-d_h}{2d_w}\Bigr)\qquad\text{and}\qquad\vartheta_x\in\Bigl(0,\min\Bigl\{\kappa,\frac{(1+2\alpha)d_w-d_h}{2}\Bigr\}\Bigr)\,.
\end{equation*}
There exists a constant $C$, depending on $p,T,T',\vartheta_t,\vartheta_x,\beta,\|u_0\|_{L^\infty}$, the structural constants and the constants in~\eqref{eq:wBE}, such that for all $s,t\in[T,T']$ and $x,y\in X$,
\begin{equation}\label{eq:increment-moments}
\EE\bigl[|u(t,x)-u(s,y)|^p\bigr]\le C\bigl(|t-s|^{\vartheta_tp}+d(x,y)^{\vartheta_xp}\bigr)\,.
\end{equation}
Consequently, $u$ admits a modification which is jointly continuous on $(0,\infty)\times X$ and which is, on every set $[T,T']\times B(x_0,R)$, H\"older continuous of any order $\gamma_t<((1+2\alpha)d_w-d_h)/(2d_w)$ in the time variable and of any order $\gamma_x<\min\{\kappa,((1+2\alpha)d_w-d_h)/2\}$ in the space variable.
\end{thm}

\begin{proof}
Without loss of generality we assume $s\le t$. The  bound~\eqref{eq:wBE}, together with $L^\infty$ contractivity for $d(x,y)>1$, gives $|P_tu_0(x)-P_tu_0(y)|\le Cd(x,y)^{\vartheta_x}$ on $[T,T']$. We then have
\begin{equation*}
|P_tu_0(x)-P_su_0(y)|\le|P_tu_0(x)-P_tu_0(y)|+|P_tu_0(y)-P_su_0(y)|\le C\bigl(d(x,y)^{\vartheta_x}+|t-s|^{\vartheta_t}\bigr)\,.
\end{equation*}
Indeed, for time increments with $t\le2s$, integrate~\eqref{eq:weighted-time} with $\vartheta=\vartheta_t$ against $|u_0|\,d\mu$ and use conservativeness; for $t>2s$, contractivity gives the same bound after enlarging the constant, since $(t-s)/s>1$.

For the increments of the solution, we have by~\eqref{eq:mild}, $u(t,x)-u(s,y)=P_tu_0(x)-P_su_0(y)+\beta\bigl(I(t,x)-I(s,y)\bigr)$ a.s.

It remains to estimate the stochastic integral, whose increment splits into a space increment of the kernel, a time increment of the kernel, and the integral over the short time interval $(s,t)$. Write $I(t,x)-I(s,y)=A_1+A_2+A_3$ with
\begin{align*}
A_1&=\int_0^t\!\!\int_X\bigl[p_{t-\tau}(x,z)-p_{t-\tau}(y,z)\bigr]u(\tau,z)\,W_\alpha(d\tau,dz),\\
A_2&=\int_0^s\!\!\int_X\bigl[p_{t-\tau}(y,z)-p_{s-\tau}(y,z)\bigr]u(\tau,z)\,W_\alpha(d\tau,dz),\\
A_3&=\int_s^t\!\!\int_Xp_{t-\tau}(y,z)u(\tau,z)\,W_\alpha(d\tau,dz)\,.
\end{align*}
Set
\begin{equation*}
K_{T'}:=\sup_{(\tau,z)\in(0,T']\times X}\|u(\tau,z)\|_{L^p(\Omega)}<\infty\,,
\end{equation*}
which is finite by Proposition~\ref{prop:Lp-estimates}. We first record the common estimate used for the three terms. Let $h_\tau(z)$ be a deterministic measurable function on $(0,T')\times X$ such that $\int_0^{T'}\||h_\tau|\|_{\cW^{-\alpha}}^2\,d\tau<\infty$. If $\alpha>0$, the covariance formula~\eqref{eq:H-inner-product}, Minkowski's inequality and H\"older's inequality give
\begin{align*}
\Bigl\|\|h_\tau u(\tau,\cdot)\|^2_{\cW^{-\alpha}}\Bigr\|_{L^{p/2}(\Omega)}
&\le\int_{X^2}|h_\tau(z)h_\tau(z')|G_{2\alpha}(z,z')\\
&\qquad\times\|u(\tau,z)\|_{L^p(\Omega)}\|u(\tau,z')\|_{L^p(\Omega)}\,d\mu(z)d\mu(z')\\
&\le K_{T'}^2\bigl\||h_\tau|\bigr\|^2_{\cW^{-\alpha}}\,.
\end{align*}
For $\alpha=0$, the same bound follows from~\eqref{eq:H0-inner-product} and
\begin{equation*}
\Bigl\|\|h_\tau u(\tau,\cdot)\|_{L^2(X,\mu)}^2\Bigr\|_{L^{p/2}(\Omega)}
\le\int_X|h_\tau(z)|^2\|u(\tau,z)\|_{L^p(\Omega)}^2\,d\mu(z)
\le K_{T'}^2\|h_\tau\|_{L^2(X,\mu)}^2\,.
\end{equation*}
In particular, $h_\tau u(\tau,\cdot)\in\Lambda_\alpha$ almost surely for almost every $\tau$, and its expected squared norm is integrable in time. Since $u$ is adapted, the stochastic integral is therefore well defined. As in~\cite[Proposition~6.7]{HamblyYang2026}, the Burkholder--Davis--Gundy inequality~\eqref{eq:BDG} and Minkowski's inequality now yield
\begin{align}
\left\|\int_0^{T'}\!\!\int_Xh_\tau(z)u(\tau,z)\,W_\alpha(d\tau,dz)\right\|_{L^p(\Omega)}^2
&\le C_p\int_0^{T'}\Bigl\|\|h_\tau u(\tau,\cdot)\|^2_{\cW^{-\alpha}}\Bigr\|_{L^{p/2}(\Omega)}\,d\tau\label{eq:holder-bdg-first}\\
&\le C_pK_{T'}^2\int_0^{T'}\bigl\||h_\tau|\bigr\|^2_{\cW^{-\alpha}}\,d\tau\,.\label{eq:holder-bdg-common}
\end{align}
We apply~\eqref{eq:holder-bdg-common} to the integrands defining $A_1,A_2,A_3$, extended by zero outside their respective time intervals. The following estimates also verify the required time integrability.

\emph{The term $A_1$.} Choose $\vartheta\in(0,1)$ with $\kappa\vartheta=\vartheta_x$; this is possible since $\vartheta_x<\kappa$, and then $\vartheta<((1+2\alpha)d_w-d_h)/(2\kappa)$ since $\vartheta_x<((1+2\alpha)d_w-d_h)/2$. By~\eqref{eq:W-space},
\begin{align*}
\EE\bigl[|A_1|^p\bigr]^{2/p}
&\le C\int_0^td(x,y)^{2\kappa\vartheta}(t-\tau)^{2\alpha-\frac{d_h}{d_w}-\frac{2\kappa\vartheta}{d_w}}d\tau\\
&=C\,d(x,y)^{2\vartheta_x}\int_0^tr^{2\alpha-\frac{d_h}{d_w}-\frac{2\vartheta_x}{d_w}}dr
\le C_{T'}\,d(x,y)^{2\vartheta_x}\,,
\end{align*}
the integral being finite because $1+2\alpha-\frac{d_h}{d_w}-2\vartheta_x/d_w>0$.

\emph{The term $A_2$.} By~\eqref{eq:W-time} with $\vartheta=\vartheta_t$,
\begin{align*}
\EE\bigl[|A_2|^p\bigr]^{2/p}&\le C\int_0^s(t-s)^{2\vartheta_t}(s-\tau)^{2\alpha-\frac{d_h}{d_w}-2\vartheta_t}d\tau\\
&=C\,(t-s)^{2\vartheta_t}\int_0^sr^{2\alpha-\frac{d_h}{d_w}-2\vartheta_t}dr\le C_{T'}\,(t-s)^{2\vartheta_t}\,,
\end{align*}
the integral being finite because $2\vartheta_t<1+2\alpha-\frac{d_h}{d_w}$.

\emph{The term $A_3$.} By the bound $\|p_{t-\tau}(y,\cdot)\|^2_{\cW^{-\alpha}}\le C(t-\tau)^{2\alpha-\frac{d_h}{d_w}}$ recalled in Corollary~\ref{cor:W-increments},
\begin{equation*}
\EE\bigl[|A_3|^p\bigr]^{2/p}\le C\int_s^t(t-\tau)^{2\alpha-\frac{d_h}{d_w}}d\tau=\frac{C}{1+2\alpha-\frac{d_h}{d_w}}(t-s)^{1+2\alpha-\frac{d_h}{d_w}}\le C_{T'}\,(t-s)^{2\vartheta_t}\,,
\end{equation*}
since $2\vartheta_t<1+2\alpha-\frac{d_h}{d_w}$ and $t-s\le T'$.

Collecting the bounds, raising to the power $p/2$ and using $(v_1+v_2+v_3)^{p}\le3^{p-1}(v_1^p+v_2^p+v_3^p)$ gives~\eqref{eq:increment-moments}.

For the modification, fix a compact cylinder $[T,T']\times\overline B(o,R)$ and admissible $\vartheta_t,\vartheta_x$. The metric
\begin{equation*}
\varrho\bigl((t,x),(s,y)\bigr)=|t-s|^{\vartheta_t}+d(x,y)^{\vartheta_x}
\end{equation*}
induces the product topology, and Ahlfors regularity bounds its covering numbers by $C\varepsilon^{-(1/\vartheta_t+d_h/\vartheta_x)}$. Estimate~\eqref{eq:increment-moments} holds for arbitrarily large $p$. Hence the metric-space Kolmogorov--Chentsov theorem~\cite[Theorem~1.1]{KraetschmerUrusov2023} gives a modification that is $\varrho$-H\"older of every fixed order $r<1-(1/\vartheta_t+d_h/\vartheta_x)/p$. Choosing $p$ large and $r$ close to $1$ yields every pair of orders strictly below $\vartheta_t,\vartheta_x$.

We then apply this construction to the countable exhaustion $[m^{-1},m]\times\overline B(o,m)$, $m\ge2$, and to countable sequences of admissible exponents approaching the stated limits, with $r\uparrow1$. The continuous modifications agree almost surely on countable dense subsets of their overlaps, hence everywhere on those overlaps and hence define a single modification with all the asserted local H\"older exponents.
\end{proof}

\begin{remark}
For $\alpha>0$, in the Euclidean case $d_w=2$, $\kappa=1$, the exponents of Theorem~\ref{thm:Holder} are $(2+4\alpha-d_h)/4$ in time and $(2+4\alpha-d_h)/2$ in space for the Riesz covariance $|x-y|^{4\alpha-d_h}$, $0<d_h-4\alpha<2$. These coincide with the exponents in~\cite[Theorem~1.8]{ChenHuang2019} for the equation driven by noise that is white in time and has Riesz spatial covariance. At $\alpha=0$, necessarily $d_h<d_w$, the corresponding exponents are $(d_w-d_h)/(2d_w)$ in time and $\min\{\kappa,(d_w-d_h)/2\}$ in space; for $X=\R$ these are $1/4$ and $1/2$. For compact p.c.f.\ fractals, space-time regularity was studied for the additive equation in~\cite{HamblyYang2018} and for nonlinear equations, including PAM, in~\cite{HamblyYang2026}. On standard nested fractals their white-noise exponents, expressed in the metric $d$, agree with $(d_w-d_h)/(2d_w)$ in time and $(d_w-d_h)/2$ in space. In the present theorem the space exponent is also limited by exponent $\kappa$.
\end{remark}

\section{Existence of the \texorpdfstring{$L^p$}{Lp} Lyapunov exponents}\label{sec:lyapunov}

We prove existence of the moment Lyapunov exponents for every real $p\ge1$, independently of the spatial point and of the nonzero nonnegative bounded initial condition. We first construct the fundamental solution by estimating its Wiener chaos kernels and establish the flow, time-reversal and moment properties used below. The proof of Theorem~\ref{thm:lyapunov-main} then has three steps. A heat kernel profile gives a pointwise moment quantity that is supermultiplicative after a fixed time shift. Supermultiplicativity then yields an exponential rate by standard results.  Time reversal and sub-Gaussian tails identify its growth rate with that for the constant initial condition at each spatial point. Finally, heat kernel comparisons identify the rates at different points and extend the conclusion to general nonzero nonnegative bounded initial conditions. The moment comparisons are adapted from~\cite[Lemma~5.1 and proof of Proposition~2.8]{BergerChongLacoin2023}, with heat kernel bounds in place of spatial translation invariance.

\vspace{0.4cm}

Throughout this section, $\alpha$ satisfies~\eqref{eq:alpha-range} and $\beta>0$. Unless a dependence is specified, constants may depend on $\alpha$, $\beta$, the moment order and the data in Assumption~\ref{hyp:dirichlet-form}. They are uniform in time and in the spatial variables.

\vspace{0.4cm}

The goal of the section is to prove the following theorem.

\begin{thm}\label{thm:lyapunov-main}
Let $u_0\in L^\infty(X,\mu)$ be nonnegative and nonzero. For every real $p\ge1$ and every $x\in X$, the finite limit
\begin{equation}\label{eq:mainLimit}
\lambda_p(\beta):=\lim_{t\to\infty}\frac1t\log\EE[u(t,x)^p]
\end{equation}
exists and is independent of $x$ and of the initial condition $u_0$.
\end{thm}

\subsection{Chaos estimates for the fundamental solution}\label{sec:regularization}

For  times $0\le s<t$ and bounded measurable $f$, let $U_{s,t}f(x)$ denote the solution of~\eqref{eq:pam-D} at $(t,x)$ with initial condition $f$ at time $s$, driven by the noise on $(s,t)$. Write $U_t=U_{0,t}$ and let $\mathcal G_{s,t}$ be the $\sigma$-field generated by the noise on that interval $(s,t)$. Thus $u(t,x)=U_tu_0(x)$. 

As a preliminary observation, note that Proposition~\ref{prop:Lp-estimates}, together with monotonicity of $L^q(\Omega)$ norms, gives, for every finite $q\ge1$,
\begin{equation}\label{eq:solution-moment-bound}
\norm{U_tf(x)}_{L^q(\Omega)}\le Ce^{Ct}\norm{f}_{L^\infty},\qquad t>0,\quad x\in X\,.
\end{equation}

For $k\ge1$ and $t>0$, define $g_k$ on the ordered simplex $0<a_1<\cdots<a_k<t$ by
\begin{equation}\label{eq:flow-chaos-kernel}
g_k(a,z;t,x,y)=\frac{\beta^k}{k!}\,p_{t-a_k}(x,z_k)
\prod_{i=1}^{k-1}p_{a_{i+1}-a_i}(z_{i+1},z_i)\,p_{a_1}(z_1,y)\,.
\end{equation}
Extend $g_k$ by symmetry in the pairs $(a_i,z_i)$ and set it equal to zero outside $(0,t)^k$ in the time variables. The key lemma that yields the existence of the fundamental solution for the PAM model is the following.

\begin{lem}\label{lem:kernel-chaos-estimate}
There exists $\sigma\ge1$, such that, for every  $q\ge1$, $x,y \in X$ and $t>0$,
\begin{equation}\label{eq:flow-chaos-summability}
p_t(x,y)+\sum_{k\ge1}\norm{I_k(g_k(\cdot;t,x,y))}_{L^q(\Omega)}
\le Ce^{Ct}p_{\sigma t}(x,y)\,.
\end{equation}
Here $C$ may also depend on $q$ and $\beta$.  We keep the notation $\sigma$ fixed throughout the section.
\end{lem}

\begin{proof}
We first obtain a uniform estimate by induction on the chaos order. Put $g_0(\cdot;t,x,y)=p_t(x,y)$ and interpret $\cH_\alpha^{\otimes0}$ as $\R$. For this proof, set
\begin{equation*}
M_k(t):=\sup_{x,y\in X}k!\,\norm{g_k(\cdot;t,x,y)}_{\cH_\alpha^{\otimes k}}^2,\qquad k\ge0\,.
\end{equation*}
The heat kernel upper bound gives $M_0(t)\le Ct^{-2d_h/d_w}$. We claim that
\begin{equation}\label{eq:flow-chaos-recursion}
M_{k+1}(t)\le C\beta^2\int_0^{t/2}r^{2\alpha-d_h/d_w}M_k(t-r)\,dr\,.
\end{equation}
For $\alpha>0$, symmetry in the time-space variables and the covariance formula give
\begin{align*}
\|g_{k+1}(\cdot;t,x,y)\|_{\cH_\alpha^{\otimes(k+1)}}^2
&=(k+1)!\int_{0<a_1<\cdots<a_{k+1}<t}
\int_{X^{2(k+1)}}g_{k+1}(a,z;t,x,y)g_{k+1}(a,z';t,x,y)\\
&\qquad\times\prod_{i=1}^{k+1}G_{2\alpha}(z_i,z_i')
\,d\mu^{\otimes(k+1)}(z)\,d\mu^{\otimes(k+1)}(z')\,da\,.
\end{align*}
The integrand is nonnegative because the heat kernels and $G_{2\alpha}$ are nonnegative. Moreover, on the ordered time simplex,
\begin{equation*}
1\le
\mathbf{1}_{\{a_{k+1}\ge t/2\}}
+\mathbf{1}_{\{a_1\le t/2\}}\,.
\end{equation*}
Inserting this inequality into the preceding integral bounds the squared norm by the sum of the same integral restricted respectively to $\{a_{k+1}\ge t/2\}$ and $\{a_1\le t/2\}$. These regions may overlap; their intersection is then counted twice, which is harmless for an upper bound.


We now multiply the squared-norm identity by $(k+1)!$. On the first region, put $r=t-a_{k+1}$. Integrating over the remaining $k$ times and their spatial variables gives the following contribution to $(k+1)!\,\|g_{k+1}(\cdot;t,x,y)\|_{\cH_\alpha^{\otimes(k+1)}}^2$:
\begin{align*}
&\beta^2\int_0^{t/2}\int_{X^2}p_r(x,z)p_r(x,z')G_{2\alpha}(z,z')\\
&\qquad\times k!\,\bigl\langle g_k(\cdot;t-r,z,y),g_k(\cdot;t-r,z',y)\bigr\rangle_{\cH_\alpha^{\otimes k}}
\,d\mu(z)d\mu(z')\,dr\,.
\end{align*}
By Cauchy--Schwarz inequality,
\begin{equation*}
k!\,\bigl|\bigl\langle g_k(\cdot;t-r,z,y),g_k(\cdot;t-r,z',y)\bigr\rangle_{\cH_\alpha^{\otimes k}}\bigr|
\le M_k(t-r)\,.
\end{equation*}
Thus~\eqref{eq:two-point-kernel} bounds this contribution by
\begin{align*}
&\beta^2\int_0^{t/2}M_k(t-r)
\int_{X^2}p_r(x,z)p_r(x,z')G_{2\alpha}(z,z')\,d\mu(z)d\mu(z')\,dr\\
&\qquad\le C\beta^2\int_0^{t/2}r^{2\alpha-d_h/d_w}M_k(t-r)\,dr\,.
\end{align*}
Reversing time and both chains of heat kernels, and interchanging $x,y$, transforms the second region into the first. It therefore satisfies the same bound, proving~\eqref{eq:flow-chaos-recursion}. When $\alpha=0$, each primed spatial variable coincides with its unprimed counterpart. A similar argument which uses
\begin{equation*}
\int_Xp_r(x,z)^2\,d\mu(z)=p_{2r}(x,x)\le Cr^{-d_h/d_w}
\end{equation*}
again gives~\eqref{eq:flow-chaos-recursion}.

Now, fix $\rho>0$. Multiplying~\eqref{eq:flow-chaos-recursion} by $t^{2d_h/d_w}e^{-\rho t}$ and using $t-r\ge t/2$, we obtain
\begin{align*}
\sup_{t>0}t^{2d_h/d_w}e^{-\rho t}M_{k+1}(t)
&\le C\beta^2\left(\int_0^\infty e^{-\rho r}r^{2\alpha-d_h/d_w}\,dr\right)
\sup_{s>0}s^{2d_h/d_w}e^{-\rho s}M_k(s)\\
&\le C\beta^2\rho^{d_h/d_w-1-2\alpha}
\sup_{s>0}s^{2d_h/d_w}e^{-\rho s}M_k(s)\,.
\end{align*}
 Starting from the bound for $M_0$, an induction yields
\begin{equation*}
M_k(t)\le Ct^{-2d_h/d_w}e^{\rho t}
\left(C\beta^2\rho^{d_h/d_w-1-2\alpha}\right)^k\,.
\end{equation*}

We next handle the off-diagonal  sub-Gaussian decay. By~\eqref{eq:subGauss}, we may fix $c>0$ sufficiently small and $\sigma\ge1$ sufficiently large that
\begin{equation}\label{harck:local}
p_r(a,b)\le Cp_{\sigma r}(a,b)
\exp\left[-c\left(\frac{d(a,b)^{d_w}}r\right)^{1/(d_w-1)}\right]
\qquad(r>0,\ a,b\in X)\,.
\end{equation}
Choose the same $\sigma$ so that~\eqref{eq:exp-to-kernel} holds with $c'=c$ and $c''=\sigma$. For  $r_0,\ldots,r_k>0$ with sum $t$ and a chain $z_0=y, \cdots, z_{k+1}=x$, H\"older's inequality gives
\begin{align*}
d(x,y)&\le\sum_{i=0}^k\frac{d(z_{i+1},z_i)}{r_i^{1/d_w}}\,r_i^{1/d_w}\\
&\le\left[\sum_{i=0}^k
\left(\frac{d(z_{i+1},z_i)^{d_w}}{r_i}\right)^{1/(d_w-1)}\right]^{(d_w-1)/d_w}t^{1/d_w}\,.
\end{align*}
Consequently,
\begin{equation*}
\sum_{i=0}^k\left(\frac{d(z_{i+1},z_i)^{d_w}}{r_i}\right)^{1/(d_w-1)}
\ge\left(\frac{d(x,y)^{d_w}}t\right)^{1/(d_w-1)}\,.
\end{equation*}
Let $\widetilde g_k$ be obtained from~\eqref{eq:flow-chaos-kernel} by replacing every $p_r$ by $p_{\sigma r}$. Using \eqref{harck:local} gives
\begin{equation*}
k!\,\norm{g_k}_{\cH_\alpha^{\otimes k}}^2
\le C^{2k+2}\exp\left[-2c\left(\frac{d(x,y)^{d_w}}t\right)^{1/(d_w-1)}\right]
k!\,\norm{\widetilde g_k}_{\cH_\alpha^{\otimes k}}^2\,.
\end{equation*}
The preceding induction also applies with $p_r$ replaced by $p_{\sigma r}$.
%
Absorbing $C^{2k+2}$ into the constants therefore yields
\begin{equation}\label{eq:flow-chaos-L2}
\begin{aligned}
k!\,\norm{g_k(\cdot;t,x,y)}_{\cH_\alpha^{\otimes k}}^2
&\le Ct^{-2d_h/d_w}e^{\rho t}
\left(C\beta^2\rho^{d_h/d_w-1-2\alpha}\right)^k\\
&\qquad\times\exp\left[-2c\left(\frac{d(x,y)^{d_w}}t\right)^{1/(d_w-1)}\right]\,.
\end{aligned}
\end{equation}

As before, for $q\ge2$, hypercontractivity on the $k$-th Wiener chaos gives
\begin{equation*}
\norm{I_k(g_k)}_{L^q(\Omega)}\le(q-1)^{k/2}\sqrt{k!}\,\norm{g_k}_{\cH_\alpha^{\otimes k}}.
\end{equation*}
 Since $d_h/d_w-1-2\alpha<0$, choose $\rho>0$ so that $C\beta^2(q-1)\rho^{d_h/d_w-1-2\alpha}\le1/4$. Taking square roots in~\eqref{eq:flow-chaos-L2} and summing, we obtain
\begin{align*}
\sum_{k\ge1}\norm{I_k(g_k)}_{L^q(\Omega)}
&\le Ct^{-d_h/d_w}e^{\rho t/2}
\exp\left[-c\left(\frac{d(x,y)^{d_w}}t\right)^{1/(d_w-1)}\right]\sum_{k\ge1}2^{-k}\\
&\le Ce^{\rho t/2}p_{\sigma t}(x,y)\,.
\end{align*}
The heat kernel comparison also bounds $p_t(x,y)$ by $Cp_{\sigma t}(x,y)$, proving~\eqref{eq:flow-chaos-summability}. The range $1\le q<2$ follows from the case $q=2$ by monotonicity of $L^q(\Omega)$ norms.
\end{proof}

\subsection{The fundamental solution and its flow properties}\label{sec:fundamental-solution}

We now prove existence of the fundamental solution and  establish the flow properties used in the proof of Theorem~\ref{thm:lyapunov-main}. Their one-dimensional counterparts are stated in~\cite[Theorem~3.1]{AlbertsKhaninQuastel2014}.

\begin{thm}\label{lem:regularized-flow}
For every  $0\le s<t$, the solution operator $U_{s,t}$ has the following properties.
\begin{enumerate}
\item There is a nonnegative random kernel $\mathcal{Z}_{s,t}(x,y)$, jointly measurable in $(\omega,x,y)$ and $\mathcal G_{s,t}$-measurable in $\omega$, such that, for every bounded nonnegative measurable $f$ and every $x\in X$,
\begin{equation}\label{eq:flow-kernel}
U_{s,t}f(x)=\int_X\mathcal{Z}_{s,t}(x,y)f(y)\,d\mu(y)\quad\text{a.s.}
\end{equation}
The same formula defines $U_{s,t}f$ for every nonnegative measurable $f$, with values in $[0,\infty]$, and satisfies
\begin{equation}\label{eq:flow-mean}
\EE[U_{s,t}f(x)]=P_{t-s}f(x)\,.
\end{equation}
\item For every finite family of pairwise disjoint deterministic intervals $(s_j,t_j)$, the kernels $\mathcal{Z}_{s_j,t_j}$ are independent. Moreover, $\mathcal{Z}_{s,t}$ has the law of $\mathcal{Z}_{0,t-s}$, and hence $U_{s,t}$ has the law of $U_{t-s}$.
\item For deterministic times $0\le r<s<t$, every nonnegative measurable $f$ and each fixed $x\in X$,
\begin{equation}\label{eq:flow-composition}
U_{r,t}f(x)=U_{s,t}(U_{r,s}f)(x)\quad\text{a.s.}
\end{equation}
For every fixed pair $(x,y)\in X^2$, almost surely,
\begin{equation}\label{eq:flow-kernel-composition}
\mathcal{Z}_{r,t}(x,y)=\int_X\mathcal{Z}_{s,t}(x,z)\mathcal{Z}_{r,s}(z,y)\,d\mu(z)\,.
\end{equation}
For each fixed $x$, this identity also holds almost surely for $\mu$-almost every $y$.
\item The transposed kernel $\mathcal{Z}_{s,t}^*(x,y)=\mathcal{Z}_{s,t}(y,x)$ satisfies
\begin{equation}\label{eq:flow-reversal}
\mathcal{Z}_{s,t}^*\stackrel{\mathrm{law}}=\mathcal{Z}_{s,t}
\end{equation}
as random kernels modulo $\mu\otimes\mu$-null sets. Moreover, for each fixed $x\in X$, the random functions $\mathcal Z_{s,t}(x,\cdot)$ and $\mathcal Z_{s,t}(\cdot,x)$ have the same law modulo $\mu$-null sets.
\item For every finite $q\ge1$, there is $C>0$, depending on $q$, $\beta$, $\alpha$ and the geometric data, such that, for every $x,y\in X$,
\begin{equation}\label{eq:flow-kernel-moment-bound}
\norm{\mathcal Z_{s,t}(x,y)}_{L^q(\Omega)}\le Ce^{C(t-s)}p_{\sigma(t-s)}(x,y)
\end{equation}
Consequently, for every nonnegative measurable $f$ and every $x\in X$,
\begin{equation}\label{eq:heat-kernel-smoothing}
\norm{U_{s,t}f(x)}_{L^q(\Omega)}\le Ce^{C(t-s)}P_{\sigma(t-s)}f(x)\,.
\end{equation}
\end{enumerate}
\end{thm}

\begin{proof}
We first construct the kernel and establish its representation, time-reversal and composition identities. We then use positivity of the solution to choose a nonnegative version. 

\emph{Construction and representation.}
For $s=0$, define
\begin{equation}\label{eq:flow-kernel-chaos}
\mathcal Z_{0,t}(x,y)=p_t(x,y)+\sum_{k\ge1}I_k(g_k(\cdot;t,x,y))\,.
\end{equation}
For general $s<t$, use the same construction with $t-s$ and noise times shifted by $s$. The kernels depend measurably on $(x,y)$ as elements of the separable Hilbert space $\cH_\alpha^{\otimes k}$. Their Wiener integrals therefore admit jointly measurable versions that are $\mathcal G_{s,t}$-measurable in $\omega$. Estimate~\eqref{eq:flow-chaos-summability} gives absolute convergence in every $L^q(\Omega)$ and almost surely for each fixed pair $(x,y)$. Taking the sum wherever it converges absolutely, and zero elsewhere, defines a jointly measurable random kernel.

For every bounded $f$,~\eqref{eq:flow-chaos-summability} and stochastic completeness give
\begin{equation*}
\int_X\left[p_t(x,y)+\sum_{k\ge1}\norm{I_k(g_k(\cdot;t,x,y))}_{L^2(\Omega)}\right]|f(y)|\,d\mu(y)
\le Ce^{Ct}\norm{f}_{L^\infty}\,.
\end{equation*}
Hence the integral in space and the chaos series may be interchanged in $L^2(\Omega)$. The Wiener isometry also allows spatial integration inside each multiple integral. Integrating~\eqref{eq:flow-chaos-kernel} against $f(y)$ replaces $p_{a_1}(z_1,y)$ by $P_{a_1}f(z_1)$, recovering the solution kernel~\eqref{eq:fk}. Comparison with~\eqref{eq:chaos-expansion} proves~\eqref{eq:flow-kernel}. Taking expectation in the series defining  $\mathcal Z_{s,t}$ gives
\begin{equation}\label{eq:flow-kernel-mean}
\EE[\mathcal Z_{s,t}(x,y)]=p_{t-s}(x,y),\qquad x,y\in X\,.
\end{equation}

\emph{Time reversal.}
For functions supported in $(s,t)\times X$, put
\begin{equation*}
(R\varphi)(a,z)=\varphi(s+t-a,z),\qquad
W_\alpha^{\leftarrow}(\varphi)=W_\alpha(R\varphi)\,.
\end{equation*}
Time reflection is an isometry of $L^2((s,t);\cW^{-\alpha})$, so $W_\alpha^{\leftarrow}$ has the same law as $W_\alpha$ restricted to $(s,t)$. Reversing the order of the variables in~\eqref{eq:flow-chaos-kernel} and using symmetry of $p_t$ exchanges $x$ and $y$. The identity $I_k^{W_\alpha^{\leftarrow}}(h)=I_k^{W_\alpha}(R^{\otimes k}h)$, first on elementary tensors and then by the Wiener isometry, therefore gives
\begin{equation*}
\mathcal Z_{s,t}^{\leftarrow}(x,y)=\mathcal Z_{s,t}(y,x)\quad\text{a.s. for each fixed }x,y\,.
\end{equation*}
By Fubini's theorem, the identity holds almost surely modulo $\mu\otimes\mu$-null sets, proving~\eqref{eq:flow-reversal}. For each fixed $x$, Fubini's theorem also gives the identity almost surely for $\mu$-almost every $y$, which proves the equality in law of the corresponding row and column. For the related convolution identity in one dimension, see~\cite[Proposition~1.6 and Section~4.3]{GhosalLin2023}.
\emph{Composition.}
We adapt the uniqueness argument of~\cite[proof of Theorem~3.1(vii)]{AlbertsKhaninQuastel2014}. Fix $0\le r<s$ and $y\in X$. For $t>s$, the fields
\begin{equation*}
\mathcal Z_{r,t}(x,y)
\qquad\text{and}\qquad
\int_X\mathcal Z_{s,t}(x,z)\mathcal Z_{r,s}(z,y)\,d\mu(z)
\end{equation*}
satisfy the same mild equation on $(s,\infty)$, with initial condition $\mathcal Z_{r,s}(\cdot,y)$ and the same noise. For the first field, this follows by splitting its stochastic convolution at $s$; for the second, by linearity and stochastic Fubini. These operations are justified by~\eqref{eq:flow-chaos-summability} and independence on disjoint time intervals. In particular,
\begin{equation*}
\int_X
\norm{\mathcal Z_{s,t}(x,z)}_{L^2(\Omega)}
\norm{\mathcal Z_{r,s}(z,y)}_{L^2(\Omega)}\,d\mu(z)
\le Ce^{C(t-r)}p_{\sigma(t-r)}(x,y)\,.
\end{equation*}
Since $s-r>0$, both fields have uniformly bounded second moments for $s<t\le T$, $x\in X$, for each $T>s$. The exponentially weighted estimate in the proof of Lemma~\ref{lem:expansion-l2-estimates}, applied to their difference with zero initial condition, gives uniqueness upon choosing $\rho$ so that $C\beta^2\rho^{d_h/d_w-1-2\alpha}<1$. This proves~\eqref{eq:flow-kernel-composition} in $L^2(\Omega)$ for each fixed pair $(x,y)$.

\emph{Positivity and the choice of version.}
For each fixed $x$, the kernel $\mathcal Z_{s,t}(x,\cdot)$ is integrable almost surely by~\eqref{eq:flow-chaos-summability}. Testing~\eqref{eq:flow-kernel} against indicators from a countable algebra generating $\mathcal B(X)$, Theorem~\ref{thm:exist-unique-wholespace} and the monotone class theorem give $\mathcal Z_{s,t}(x,y)\ge0$ for $\mu$-almost every $y$, almost surely. Time reversal gives the corresponding assertion for each fixed column, and composition at the midpoint yields $\mathcal Z_{s,t}(x,y)\ge0$ almost surely for every fixed pair $(x,y)$.

Replacing the kernel by its positive part therefore preserves every fixed-endpoint random variable. Fubini's theorem preserves the integral identities, with the exceptional sets specified in the statement. Formula~\eqref{eq:flow-kernel} now defines $U_{s,t}f$ for every nonnegative measurable $f$, and Tonelli's theorem gives~\eqref{eq:flow-mean} and~\eqref{eq:flow-composition}. Dependence only on the noise in $(s,t)$ gives independence on disjoint intervals, while time-translation invariance gives stationarity; see~\cite[proof of Theorem~3.1(ii), (vi)]{AlbertsKhaninQuastel2014}. Taking the positive part also preserves time reversal.

\emph{Moment estimates.}
The chaos expansion~\eqref{eq:flow-kernel-chaos}, the triangle inequality and~\eqref{eq:flow-chaos-summability} give~\eqref{eq:flow-kernel-moment-bound} when $s=0$; stationarity gives the general case. The kernel representation~\eqref{eq:flow-kernel} and Minkowski's integral inequality then prove~\eqref{eq:heat-kernel-smoothing} for bounded nonnegative $f$. Monotone approximation extends the estimate to every nonnegative measurable $f$, with infinite right-hand sides allowed.
\end{proof}

\subsection{Proof of existence of the Lyapunov exponents}\label{sec:real-orders}

\begin{proof}[Proof of Theorem~\ref{thm:lyapunov-main}]
Fix a real $p\ge1$ and an arbitrary reference point $o\in X$. Let $\sigma$ be as in~\eqref{eq:flow-kernel-moment-bound}, and set
\begin{equation*}
f(y)=p_\sigma(o,y),\qquad a(t)=\EE[(U_tf(o))^p]\,.
\end{equation*}
The function $f$ is bounded and positive. 

\vspace{0.4cm}

\emph{Step 1: a pointwise supermultiplicative inequality.} By~\eqref{eq:flow-kernel-moment-bound}, fix $C_0>0$ such that $\|\mathcal Z_{0,1}(x,y)\|_{L^p(\Omega)}\le C_0p_\sigma(x,y)$ for every $x,y\in X$. Minkowski's inequality then gives
\begin{equation}\label{eq:profile-smoothing}
\norm{U_1g(o)}_{L^p(\Omega)}\le C_0P_\sigma g(o)
=C_0\int_Xf(y)g(y)\,d\mu(y)\,,
\end{equation}
for every nonnegative measurable $g$. If $g$ is random and independent of the noise on $[0,1]$, the same estimate holds conditionally on $g$, first for bounded truncations and then by monotone convergence for arbitrary functions. In particular, independence and the composition $U_{1+s}=U_{s,1+s}U_s$ give
\begin{equation}\label{eq:profile-moment-smoothing}
a(s+1)\le C_0^p\EE\left[\bigl(P_\sigma U_sf(o)\bigr)^p\right],\qquad s>0\,.
\end{equation}

We now observe that one can find $L>1$ and $c>0$ such that
\begin{equation}\label{eq:profile-minorization}
p_L(z,y)\ge c f(z)f(y),\qquad z,y\in X\,.
\end{equation}
Indeed, writing $\gamma=d_w/(d_w-1)$, the triangle inequality gives
\begin{equation*}
d(z,y)^\gamma\le2^{d_w/(d_w-1)-1}\bigl(d(z,o)^{d_w/(d_w-1)}+d(o,y)^{d_w/(d_w-1)}\bigr)\,,
\end{equation*}
so that  ~\eqref{eq:profile-minorization} follows from the heat kernel estimates.

Split the flow $U_{s+L+t}$ in the intervals $[0,s], [s,s+L],[s+L,s+L+t]$ of lengths $s,L,t$ and condition on $\mathcal A=\mathcal G_{0,s}\vee\mathcal G_{s+L,s+L+t}$. The middle operator $U_{s,s+L}$ is independent of $\mathcal A$ and has mean $P_L$. Conditional Tonelli's theorem and~\eqref{eq:profile-minorization} give
\begin{align*}
\EE[U_{s+L+t}f(o)\mid\mathcal A]
&=\bigl(U_{s+L,s+L+t}P_LU_sf\bigr)(o)\\
&\ge c\,\bigl(U_{s+L,s+L+t}f\bigr)(o)\,P_\sigma U_sf(o)\,.
\end{align*}
The two factors on the right are independent. Conditional Jensen's inequality, stationarity and~\eqref{eq:profile-moment-smoothing} therefore yield
\begin{equation*}
a(s+L+t)\ge c^pa(t)\EE\left[\bigl(P_\sigma U_sf(o)\bigr)^p\right]
\ge(c/C_0)^pa(t)a(s+1)\,.
\end{equation*}
Consequently,
\begin{equation}\label{eq:compact-supermultiplicativity}
a(s+t+L-1)\ge(c/C_0)^pa(s)a(t),\qquad s,t>1\,.
\end{equation}
On the other hand, Jensen's inequality, the mean identity~\eqref{eq:flow-mean} and~\eqref{eq:solution-moment-bound} give
\begin{equation*}
p_{t+\sigma}(o,o)^p\le a(t)\le Ce^{Ct}\,.
\end{equation*}
Thus $\log a$ is locally bounded below, and the diagonal heat kernel lower bound gives $a(t)\ge ct^{-pd_h/d_w}$ for $t\ge1$. Applying then the continuous-parameter Fekete argument~\cite[Lemma~2, p.~993]{Matkowski1993} to the supermultiplicative function $t\mapsto(c/C_0)^pa(t-L+1)$ on $(L,\infty)$ yields the existence of
\begin{equation}\label{eq:compact-rate}
\ell_p(o):=\lim_{t\to\infty}\frac{\log a(t)}t\in[0,\infty)\,.
\end{equation}

\emph{Step 2: comparison with the constant initial condition.} We first record an inequality for independent copies $F,F'$ of a nonnegative jointly measurable random function on a measurable set $D$ with $0<\mu(D)<\infty$. If $\int_DF\,d\mu\in L^p(\Omega)$, then
\begin{equation}\label{eq:independent-pair-inequality}
\Bigl\|\int_DFF'\,d\mu\Bigr\|_{L^p(\Omega)}
\ge\mu(D)^{-1}\Bigl\|\int_DF\,d\mu\Bigr\|_{L^p(\Omega)}^2\,.
\end{equation}
Indeed, set $M=\int_DF\,d\mu$ and $M'=\int_DF'\,d\mu$. Tonelli's theorem, independence and Cauchy--Schwarz give
\begin{align*}
\EE\Bigl[\Bigl(\int_DFF'\,d\mu\Bigr)M^{p-1}(M')^{p-1}\Bigr]
&=\int_D\bigl(\EE[F(y)M^{p-1}]\bigr)^2\,d\mu(y)\\
&\ge\mu(D)^{-1}\bigl(\EE[M^p]\bigr)^2\,.
\end{align*}
For $p>1$, H\"older's inequality bounds the left-hand side by $\|\int_DFF'\,d\mu\|_{L^p(\Omega)}(\EE[M^p])^{2(p-1)/p}$, proving~\eqref{eq:independent-pair-inequality} when $\EE[M^p]>0$ and the left-hand side of~\eqref{eq:independent-pair-inequality} is finite; the other cases are immediate. For $p=1$, omit the factors $M^{p-1}(M')^{p-1}$ in the calculation.

We now come back to the proof. The composition identity~\eqref{eq:flow-kernel-composition} gives
\begin{equation*}
\mathcal Z_{0,2t}(o,o)=\int_X\mathcal Z_{t,2t}(o,z)\mathcal Z_{0,t}(z,o)\,d\mu(z)\,.
\end{equation*}
By independence, time reversal and stationarity, the two factors are independent random functions of $z$ with the same law as $\mathcal Z_{0,t}(o,z)$, modulo $\mu$-null sets. Restricting the integral to $D$ and applying~\eqref{eq:independent-pair-inequality} yields
\begin{equation*}
\EE[(U_t\1_D(o))^p]
\le\mu(D)^{p/2}\EE[\mathcal Z_{0,2t}(o,o)^p]^{1/2}\,.
\end{equation*}

We then use:
\begin{equation*}
\mathcal Z_{0,T+1}(o,o)=\int_X\mathcal Z_{1,T+1}(o,z)\mathcal Z_{0,1}(z,o)\,d\mu(z)\,.
\end{equation*}
Condition on $\mathcal G_{1,T+1}$ and apply Minkowski's inequality to the independent first interval $[0,1]$. The kernel moment bound~\eqref{eq:flow-kernel-moment-bound} gives $\|\mathcal Z_{0,1}(z,o)\|_{L^p(\Omega)}\le C_0f(z)$. Therefore
\begin{equation}\label{eq:diagonal-profile-comparison}
\EE[\mathcal Z_{0,T+1}(o,o)^p]
\le C_0^p\EE[(U_{1,T+1}f(o))^p]=C_0^pa(T)\,.
\end{equation}
Combining the last two estimates, for $t>1/2$ we obtain
\begin{equation}\label{eq:compact-unrestricted-comparison}
\EE[(U_t\1_D(o))^p]\le C\mu(D)^{p/2}a(2t-1)^{1/2}\,.
\end{equation}

Let $D_t=B(o,rt)$, where $r$ is a sufficiently large fixed constant. The tail bound~\eqref{eq:oneParticleTail} at time $\sigma t$ and~\eqref{eq:heat-kernel-smoothing} give
\begin{equation*}
\norm{U_t\1_{X\setminus D_t}(o)}_{L^p(\Omega)}
\le C\exp\bigl((C-cr^{d_w/(d_w-1)})t\bigr)\le Ce^{-t},\qquad t\ge1\,.
\end{equation*}
Since $a(t)\le\norm{f}_{L^\infty}^p\EE[(U_t\1(o))^p]$, Ahlfors regularity,~\eqref{eq:compact-unrestricted-comparison} and Minkowski's inequality give
\begin{equation}\label{eq:averaged-flat-rate}
\norm{f}_{L^\infty}^{-p}a(t)
\le\EE[(U_t\1(o))^p]
\le Ct^{pd_h/2}a(2t-1)^{1/2}+Ce^{-pt},\qquad t\ge2\,.
\end{equation}
Taking logarithms, dividing by $t$ and using~\eqref{eq:compact-rate} and $\ell_p(o)\ge0$, we obtain
\begin{equation*}
\lim_{t\to\infty}\frac1t\log\EE[(U_t\1(o))^p]=\ell_p(o)\,.
\end{equation*}
The reference point $o$ was arbitrary, so this limit exists at every point of $X$.

\emph{Step 3: independence of $o$ and initial condition.} Fix $x\in X$. For a sufficiently large fixed $\eta$, the two-sided heat kernel bounds give
\begin{equation*}
p_\eta(x,y)\ge c_{x,o}p_\sigma(o,y),\qquad y\in X\,.
\end{equation*}
 We deduce:
\begin{align}\label{eq:solution-pointwise-average}
\EE[(U_{t+\eta}\1(x))^p]
&\ge\EE[(P_\eta U_t\1(x))^p]\notag\\
&\ge c_{x,o}^p\EE[(P_\sigma U_t\1(o))^p]\notag\\
&\ge(c_{x,o}/C_0)^p\EE[(U_{t+1}\1(o))^p]\,.
\end{align}
Indeed, conditioning on $\mathcal G_{0,t}$, the mean identity and Jensen's inequality give the first inequality. The last follows by applying~\eqref{eq:profile-smoothing} conditionally to $U_t\1$.
The limits established in Step~2 imply $\ell_p(x)\ge\ell_p(o)$. Interchanging $x$ and $o$ proves equality. Therefore $\ell_p(o)$ is independent of $o$.

Finally let $u_0$ be bounded, nonnegative and nonzero. Choose $b>0$ and a bounded measurable set $A$ of positive measure such that $u_0\ge b\1_A$ almost everywhere. For a sufficiently large fixed $\eta$, the same heat kernel comparison as before gives
\begin{equation*}
P_\eta u_0\ge c f\,.
\end{equation*}
To see this, choose $R<\infty$ with $A\subset B(o,R)$, use $d(y,z)\le d(y,o)+R$ for $z\in A$ in the lower bound for $p_\eta(y,z)$, and integrate over $A$. Conditioning on $\mathcal G_{\eta,t+\eta}$ replaces the initial operator $U_{0,\eta}$ by $P_\eta$. Jensen's inequality, stationarity and positivity therefore yield
\begin{equation}\label{eq:solution-outer-intervals}
\EE[(U_{t+\eta}u_0(o))^p]
\ge\EE[(U_t(P_\eta u_0)(o))^p]
\ge c^pa(t)\,.
\end{equation}
Together with
\begin{equation*}
\EE[(U_tu_0(o))^p]\le\norm{u_0}_{L^\infty}^p\EE[(U_t\1(o))^p]\,,
\end{equation*}
this proves that the moment growth rate for $u_0$ at $o$ is $\ell_p(o)$. Since $o$ was arbitrary and $\ell_p(o)$ is independent of $o$,~\eqref{eq:mainLimit} follows with this common value as $\lambda_p(\beta)$.
\end{proof}

\section{Upper and lower bounds for the Lyapunov exponents}\label{sec:moment-bounds}

In this section, we estimate the Lyapunov exponents whose existence was established in Theorem~\ref{thm:lyapunov-main}. For arbitrary bounded initial conditions, Proposition~\ref{prop:Lp-estimates} already gives the upper bound~\eqref{eq:Lp-limsup}. For the lower bound, we combine the Gaussian exponential test of~\cite[Proposition~3.20]{ChenDeyaOuyangTindel2021} with the ball exit estimate of Lemma~\ref{lem:ball-exit-time-LB}. Throughout this section, $\alpha$ satisfies~\eqref{eq:alpha-range} and $\beta>0$.

\subsection{Lower bounds}\label{sec:lower}

For a real Borel function $\psi$, define
\begin{equation}\label{eq:test-potential}
V_\psi(z)=
\begin{cases}
\displaystyle\int_XG_{2\alpha}(z,y)\psi(y)\,d\mu(y),&\alpha>0,\\[1ex]
\psi(z),&\alpha=0,
\end{cases}
\end{equation}
whenever the integral is absolutely convergent.

\begin{lem}\label{lem:Gaussian-test}
Let $\beta>0$ and let $u_0\in L^\infty(X,\mu)$ be nonnegative. Let $\psi\in\Lambda_\alpha$ be real Borel and suppose that $V_{|\psi|}$ is bounded on $X$. Then, for every real $p>1$, $t>0$ and $x\in X$,
\begin{equation}\label{eq:Gaussian-test}
\norm{u(t,x)}_{L^p(\Omega)}
\ge \exp\Bigl(-\frac{t\norm{\psi}_{\cW^{-\alpha}}^2}{2(p-1)}\Bigr)
\EE_x\Bigl[u_0(B_t)\exp\Bigl(\beta\int_0^tV_\psi(B_s)\,ds\Bigr)\Bigr]\,.
\end{equation}
\end{lem}

\begin{proof}
We adapt the argument of~\cite[Proposition~3.20]{ChenDeyaOuyangTindel2021}. Set $h_t(s,y)=\1_{(0,t)}(s)\psi(y)\in\cH_\alpha$ and
\begin{equation*}
Z_{h_t}:=\exp\Bigl(W_\alpha(h_t)-\frac12\norm{h_t}_{\cH_\alpha}^2\Bigr),
\qquad \norm{h_t}_{\cH_\alpha}^2=t\norm{\psi}_{\cW^{-\alpha}}^2\,.
\end{equation*}
The chaos expansion $Z_{h_t}=1+\sum_{k\ge1}I_k(h_t^{\otimes k})/k!$ converges in $L^2(\Omega)$. Pairing it with~\eqref{eq:chaos-expansion} gives
\begin{equation}\label{eq:Gaussian-pairing}
\EE[u(t,x)Z_{h_t}]
=P_tu_0(x)+\sum_{k=1}^\infty
\langle f_k(\cdot;t,x),h_t^{\otimes k}\rangle_{\cH_\alpha^{\otimes k}}\,.
\end{equation}
The covariance pairing in each spatial variable produces $V_\psi$, with $V_\psi=\psi$ when $\alpha=0$. Restricting the symmetric integrand to ordered times cancels the factor $1/k!$ in~\eqref{eq:fk}. The Markov property used at times $t-s_k,\ldots,t-s_1$ then gives
\begin{align*}
\langle f_k(\cdot;t,x),h_t^{\otimes k}\rangle_{\cH_\alpha^{\otimes k}}
&=\beta^k\EE_x\Bigl[u_0(B_t)
\int_{0<s_1<\cdots<s_k<t}\prod_{j=1}^kV_\psi(B_{t-s_j})\,ds\Bigr]\\
&=\frac{\beta^k}{k!}\EE_x\Bigl[u_0(B_t)\Bigl(\int_0^tV_\psi(B_s)\,ds\Bigr)^k\Bigr]\,.
\end{align*}
The second equality follows from symmetry of the time integrand and the change of variables $r=t-s$. To justify these integrations, apply Tonelli's theorem with $|\psi|$ in place of $\psi$. Since $|V_\psi|\le V_{|\psi|}$, the semigroup property bounds the absolute integral by
\begin{equation*}
\norm{u_0}_{L^\infty}
\frac{\bigl(\beta t\norm{V_{|\psi|}}_{L^\infty}\bigr)^k}{k!}\,.
\end{equation*}
This bound justifies Fubini's theorem and is summable in $k$. Summing the exponential series in~\eqref{eq:Gaussian-pairing} therefore gives
\begin{equation}\label{eq:Gaussian-FK}
\EE[u(t,x)Z_{h_t}]
=\EE_x\Bigl[u_0(B_t)\exp\Bigl(\beta\int_0^tV_\psi(B_s)\,ds\Bigr)\Bigr]\,.
\end{equation}
For $1<p<2$, $u(t,x)\in L^p(\Omega)$ since $u(t,x)\in L^2(\Omega)$; for $p\ge2$, this follows from Proposition~\ref{prop:Lp-estimates}. For $q=p/(p-1)$, direct Gaussian integration gives
\begin{equation*}
\norm{Z_{h_t}}_{L^q(\Omega)}
=\exp\Bigl(\frac{t\norm{\psi}_{\cW^{-\alpha}}^2}{2(p-1)}\Bigr)\,.
\end{equation*}
Since the right-hand side of~\eqref{eq:Gaussian-FK} is nonnegative, H\"older's inequality proves~\eqref{eq:Gaussian-test}.
\end{proof}

\begin{thm}\label{thm:moment-LB}
Let $\alpha$ satisfy~\eqref{eq:alpha-range}, let $\beta>0$, and let $\lambda_p(\beta)$ be the exponents of Theorem~\ref{thm:lyapunov-main}. There exists $c>0$, depending only on $\alpha$ and the structural constants, such that, for every real $p>1$,
\begin{equation}\label{eq:LB-global}
\lambda_p(\beta)\ge c\,p\,\bigl(\beta^2(p-1)\bigr)^{\frac{d_w}{(1+2\alpha)d_w-d_h}}\,.
\end{equation}
\end{thm}

\begin{proof}
By Theorem~\ref{thm:lyapunov-main}, it suffices to take $u_0\equiv1$. Fix $p>1$, $x\in X$ and $r>0$. We first estimate the potential and covariance norm of the indicator of $B(x,r)$. We claim there are constants $c,C>0$, depending only on $\alpha$ and the structural constants, such that
\begin{equation}\label{eq:ball-test-estimates}
\begin{aligned}
\inf_{z\in B(x,r)}V_{\1_{B(x,r)}}(z)&\ge c r^{2\alpha d_w},
\qquad \sup_{z\in X}V_{\1_{B(x,r)}}(z)\le C r^{2\alpha d_w},\\
\norm{\1_{B(x,r)}}_{\cW^{-\alpha}}^2&\le C r^{d_h+2\alpha d_w}\,.
\end{aligned}
\end{equation}
For $\alpha=0$, these follow from $V_{\1_{B(x,r)}}=\1_{B(x,r)}$ and Ahlfors regularity. Suppose $\alpha>0$. The lower bound in~\eqref{eq:Riesz-bounds}, the inequality $d(z,y)<2r$ for $z,y\in B(x,r)$, and $\mu(B(x,r))\ge cr^{d_h}$ give the first estimate. For the upper bound, split the integral at distance $r$ from $z$. On $B(x,r)\setminus B(z,r)$ it is at most $Cr^{2\alpha d_w-d_h}\mu(B(x,r))$. The integral over $B(z,r)$ is bounded by summing over dyadic annuli:
\begin{equation*}
\int_{B(z,r)}d(z,y)^{2\alpha d_w-d_h}\,d\mu(y)
\le C r^{2\alpha d_w}\sum_{j\ge0}2^{-2j\alpha d_w}
\le C_\alpha r^{2\alpha d_w}\,.
\end{equation*}
Finally, $\norm{\1_{B(x,r)}}_{\cW^{-\alpha}}^2=\int_{B(x,r)}V_{\1_{B(x,r)}}\,d\mu$ gives the last estimate in~\eqref{eq:ball-test-estimates}. In particular, $\1_{B(x,r)}\in\Lambda_\alpha$ and its potential is bounded.

Apply Lemma~\ref{lem:Gaussian-test} with $\psi=\eta\1_{B(x,r)}$, $\eta>0$, and restrict the diffusion expectation to $\{\tau_{B(x,r)}>t\}$. Estimates~\eqref{eq:ball-test-estimates} and~\eqref{eq:exit-LB} give
\begin{equation}\label{eq:white-noise-test-ball}
\norm{u(t,x)}_{L^p(\Omega)}
\ge\delta_0\exp\left(
t\left[c\beta\eta r^{2\alpha d_w}-Cr^{-d_w}
-\frac{C\eta^2r^{d_h+2\alpha d_w}}{2(p-1)}\right]\right)\,.
\end{equation}
The quadratic expression in $\eta$ in the logarithmic rate is maximized at $\eta=(c/C)\beta(p-1)r^{-d_h}$. Set $A=(c^2/(2C))\beta^2(p-1)$. Taking logarithmic rates gives
\begin{equation}\label{eq:LB-eps}
\lambda_p(\beta)\ge p\bigl(A r^{2\alpha d_w-d_h}-Cr^{-d_w}\bigr)\,.
\end{equation}
By~\eqref{eq:parameter-inequalities}, $0<d_h-2\alpha d_w<d_w$. The expression in parentheses is maximized at
\begin{equation*}
r_*=
\Bigl(\frac{d_wC}{(d_h-2\alpha d_w)A}\Bigr)^{\frac1{(1+2\alpha)d_w-d_h}}\,,
\end{equation*}
and substitution gives~\eqref{eq:LB-global}, since every radius is admissible.
\end{proof}

\begin{remark}\label{rem:compact-moment-comparison}
For $\alpha>0$ and integer $p\ge2$, Theorem~\ref{thm:moment-LB} is the analogue of~\cite[Theorem~3.16]{BCHOTW24} in the present setting.  At $\alpha=0$, the second-moment lower bound can also be obtained by the renewal argument of~\cite[Section~3.6.2]{BCHOTW24}. For the PAM on compact p.c.f.\ fractals, Hambly and Yang~\cite{HamblyYang2026} prove weak intermittency and predict moment-order growth with exponent $2+D$, where $D$ is the Hausdorff dimension in the resistance metric. On standard nested fractals, $D=d_h/(d_w-d_h)$, so this exponent agrees with $1+d_w/(d_w-d_h)$ in our whole-space bounds.
\end{remark}

\subsection{Bounds for the exponents and intermittency}\label{sec:intermittency}

Combining Theorem~\ref{thm:lyapunov-main} with Proposition~\ref{prop:Lp-estimates} and Theorem~\ref{thm:moment-LB} gives the sharp dependence of the Lyapunov exponents on $\beta$ and $p$ for $p\ge2$. The lower bound is positive for every $p>1$ and, together with the first moment, whose logarithmic rate is zero, yields strict increase of $p\mapsto\lambda_p(\beta)/p$ on $(1,\infty)$.

\begin{cor}\label{cor:lyapunov-bounds}
For every real $p\ge2$,
\begin{equation}\label{eq:lyapunov-two-sided}
cp\bigl(\beta^2(p-1)\bigr)^{\frac{d_w}{(1+2\alpha)d_w-d_h}}
\le\lambda_p(\beta)
\le Cp\bigl(\beta^2(p-1)\bigr)^{\frac{d_w}{(1+2\alpha)d_w-d_h}}\,,
\end{equation}
where $c,C>0$ depend only on $\alpha$ and the geometric data. The lower bound also holds for every real $p>1$. For every nonzero nonnegative $u_0\in L^\infty(X,\mu)$ and every $x\in X$,
\begin{equation}\label{eq:first-moment-exponent}
\lambda_1(\beta):=\lim_{t\to\infty}\frac1t\log\EE[u(t,x)]=0\,,
\end{equation}
the function $p\mapsto\lambda_p(\beta)/p$ is strictly increasing on $(1,\infty)$, and
\begin{equation}\label{eq:strict-intermittency}
\lambda_1(\beta)<\frac{\lambda_2(\beta)}2<\frac{\lambda_3(\beta)}3<\cdots,
\qquad
\lim_{p\to\infty}\frac{\lambda_p(\beta)}p=+\infty\,.
\end{equation}
Thus the solution is fully intermittent at integer moment orders in the sense of~\cite[Section~1.2]{TindelAarhus2026}.
\end{cor}

\begin{proof}
The bounds~\eqref{eq:lyapunov-two-sided} follow from Proposition~\ref{prop:Lp-estimates}, Theorem~\ref{thm:moment-LB} and Theorem~\ref{thm:lyapunov-main}. By the independence of the initial condition in Theorem~\ref{thm:lyapunov-main}, we may take $u_0\equiv1$ when computing $\lambda_1(\beta)$. Since $\EE[U_t\1(x)]=P_t\1(x)=1$, we obtain~\eqref{eq:first-moment-exponent}.

We use the convexity argument for moment exponents recalled in~\cite[Section~1]{FoondunKhoshnevisan2009}. For real $1<p<r$, H\"older's inequality gives
\begin{equation*}
\EE[u(t,x)^p]
\le\EE[u(t,x)]^{(r-p)/(r-1)}
\EE[u(t,x)^r]^{(p-1)/(r-1)}\,.
\end{equation*}
Taking logarithmic rates and using $\lambda_1(\beta)=0$ yields
$\lambda_p(\beta)\le (p-1)\lambda_r(\beta)/(r-1)$.
Since $\lambda_p(\beta)>0$ by Theorem~\ref{thm:moment-LB},
\begin{equation*}
\frac{\lambda_r(\beta)}r
\ge\frac{r-1}{r(p-1)}\lambda_p(\beta)
>\frac{\lambda_p(\beta)}p\,.
\end{equation*}
The inequality $\lambda_1(\beta)<\lambda_2(\beta)/2$ and the divergence in~\eqref{eq:strict-intermittency} follow from~\eqref{eq:lyapunov-two-sided} and $\lambda_1(\beta)=0$.
\end{proof}

\section{Example: The infinite Sierpi\'nski gasket}\label{sec:infinite-gasket}

We conclude with a concrete example to illustrate the results. 

\subsection{The PAM on the gasket}\label{sec:gasket-geometry}

Let $q_0=(0,0)$, $q_1=(1,0)$ and $q_2=(1/2,\sqrt{3}/2)$. The compact Sierpi\'nski gasket is the unique nonempty compact set $K\subset\R^2$ such that
\begin{equation*}
K=\bigcup_{i=0}^2F_i(K),\qquad F_i(x)=\frac{x+q_i}{2}\,.
\end{equation*}
We consider the (one-sided) infinite gasket obtained by dilation about $q_0$:
\begin{equation*}
X=\bigcup_{m\ge0}2^mK\,.
\end{equation*}
We equip $X$ with its shortest-path metric $d$, which is comparable to the Euclidean distance, and with the Hausdorff measure $\mu$ normalized by $\mu(K)=1$. In particular, $2X=X$, and for every $n\in\mathbb Z$ and every Borel set $A\subset X$,
\begin{equation}\label{eq:sg-measure-scaling}
d(2^nx,2^ny)=2^nd(x,y),\qquad \mu(2^nA)=3^n\mu(A)\,.
\end{equation}
The canonical conservative Dirichlet form is obtained by extending the standard self-similar energy on $K$ to these dilated copies; its construction and the associated Brownian motion are described in~\cite{BarlowPerkins1988,Barlow1998,Kigami2001}. Its heat kernel satisfies Assumption~\ref{hyp:dirichlet-form} with
\begin{equation*}
d_h=\frac{\log3}{\log2},\qquad d_w=\frac{\log5}{\log2}\,,
\end{equation*}
and the exact scaling identity
\begin{equation}\label{eq:sg-heat-scaling}
p_{5^nt}(2^nx,2^ny)=3^{-n}p_t(x,y),\qquad n\in\mathbb Z,\ t>0,\ x,y\in X\,.
\end{equation}
The weak Bakry--\'Emery estimate~\eqref{eq:wBE} holds with $\kappa=d_w-d_h=\log(5/3)/\log2$; see~\cite[Theorem~3.7 and Section~5.1]{Besov3}.

Since $d_h<d_w$, the admissible range~\eqref{eq:alpha-range} becomes
\begin{equation}\label{eq:sg-alpha-range}
0\le\alpha<\frac{\log3}{2\log5}\,.
\end{equation}
Thus space--time white noise is included. For $\alpha>0$, the spatial covariance satisfies $G_{2\alpha}(x,y)\simeq d(x,y)^{2\alpha d_w-d_h}$. For every $\alpha$ in~\eqref{eq:sg-alpha-range}, every $\beta>0$ and every bounded initial condition, Theorem~\ref{thm:exist-unique-wholespace} gives a unique mild solution, with the finite-time moment bounds of Proposition~\ref{prop:Lp-estimates}. It is almost surely positive at every fixed positive time and spatial point when the initial condition is nonnegative and nonzero in $L^\infty(X,\mu)$.

By Theorem~\ref{thm:Holder}, the solution has a modification which is locally H\"older continuous on $(0,\infty)\times X$ with any respective time and space exponents satisfying
\begin{align*}
0<\gamma_t&<\frac{(1+2\alpha)\log5-\log3}{2\log5}\,,\\
0<\gamma_x&<\min\biggl\{\frac{\log(5/3)}{\log2},\frac{(1+2\alpha)\log5-\log3}{2\log2}\biggr\}\,.
\end{align*}
For nonzero nonnegative bounded initial conditions, Theorem~\ref{thm:lyapunov-main} and Corollary~\ref{cor:lyapunov-bounds} give Lyapunov exponents independent of the spatial point and the initial condition, with
\begin{equation}\label{eq:sg-lyapunov-bounds}
cp\bigl(\beta^2(p-1)\bigr)^{\frac{\log5}{(1+2\alpha)\log5-\log3}}
\le\lambda_p(\beta)\le
Cp\bigl(\beta^2(p-1)\bigr)^{\frac{\log5}{(1+2\alpha)\log5-\log3}},\qquad p\ge2\,.
\end{equation}

\subsection{Scaling of the solution and of the Lyapunov exponents}\label{sec:gasket-scaling}

For fixed $\alpha$ in~\eqref{eq:sg-alpha-range}, write $u^{\beta,u_0}$ for the solution with noise intensity $\beta$ and initial condition $u_0$. The dilation $x\mapsto2x$ multiplies diffusion time by $5$ and changes the effective noise intensity. The following identity makes this relation precise.

\begin{prop}\label{prop:sg-solution-scaling}
For every $n\in\mathbb Z$, $\beta>0$ and $u_0\in L^\infty(X,\mu)$,
\begin{equation}\label{eq:sg-solution-scaling}
u^{\beta,u_0}(5^nt,2^nx)\overset{\mathrm{law}}=
u^{\,\beta(5^{1+2\alpha}/3)^{n/2},\,u_0(2^n\cdot)}(t,x)
\end{equation}
as random fields indexed by $(t,x)\in(0,\infty)\times X$. In particular, for every real $p\ge1$,
\begin{equation}\label{eq:sg-lyapunov-scaling}
\lambda_p\biggl(\beta\Bigl(\frac{5^{1+2\alpha}}3\Bigr)^{n/2}\biggr)=5^n\lambda_p(\beta)\,.
\end{equation}
\end{prop}

\begin{proof}
We rescale the covariance and then the mild equation, as in the noise-scaling argument of~\cite[Lemma~3.3 and Proposition~3.4]{BOTW23}. For $\alpha>0$, changing variables $s=5^nr$ in~\eqref{eq:Riesz-intro} and using~\eqref{eq:sg-heat-scaling} gives
\begin{equation}\label{eq:sg-riesz-scaling}
G_{2\alpha}(2^nx,2^ny)=5^{2\alpha n}3^{-n}G_{2\alpha}(x,y),\qquad x\ne y\,.
\end{equation}
Define a rescaled noise, initially on bounded compactly supported test functions, by
\begin{equation*}
\widetilde W_\alpha(\varphi)
=5^{-n(1+2\alpha)/2}3^{-n/2}
W_\alpha\bigl((s,y)\mapsto\varphi(5^{-n}s,2^{-n}y)\bigr)\,.
\end{equation*}
For $\alpha>0$, the change of variables in the covariance~\eqref{eq:H-inner-product} contributes a factor $5^n3^{2n}5^{2\alpha n}3^{-n}=5^{n(1+2\alpha)}3^n$. The prefactor therefore gives
\begin{equation*}
\EE[\widetilde W_\alpha(\varphi)\widetilde W_\alpha(\psi)]
=\langle\varphi,\psi\rangle_{\cH_\alpha}\,.
\end{equation*}
At $\alpha=0$, the same identity follows directly from~\eqref{eq:H0-inner-product}, since $ds\,d\mu(y)$ scales by $5^n3^n$. By completion, $\widetilde W_\alpha$ is an isonormal Gaussian process on $\cH_\alpha$ with the same law as $W_\alpha$.

Set $v(t,x)=u^{\beta,u_0}(5^nt,2^nx)$. Applying~\eqref{eq:sg-heat-scaling} to the deterministic and stochastic terms of~\eqref{eq:mild} gives
\begin{equation*}
v(t,x)=P_t[u_0(2^n\cdot)](x)
+\beta\Bigl(\frac{5^{1+2\alpha}}3\Bigr)^{n/2}
\int_0^t\!\!\int_Xp_{t-s}(x,y)v(s,y)\,\widetilde W_\alpha(ds,dy)\,.
\end{equation*}
The stochastic change of variables holds first for elementary predictable integrands and then by the isometry. Uniqueness in Theorem~\ref{thm:exist-unique-wholespace}, or equivalently the chaos expansion, proves~\eqref{eq:sg-solution-scaling}. Taking $u_0\equiv1$, passing to $p$-th moments and dividing their logarithms by $t$ proves~\eqref{eq:sg-lyapunov-scaling} by Theorem~\ref{thm:lyapunov-main}.
\end{proof}

\begin{remark}\label{rem:lyapunov-scaling}
For every fixed $p\ge1$, the function
\begin{equation*}
s\longmapsto
\exp\biggl(-\frac{2d_w}{(1+2\alpha)d_w-d_h}\,s\biggr)\lambda_p(e^s)
\end{equation*}
is periodic with period $\frac12\bigl((1+2\alpha)\log5-\log3\bigr)$. This follows from~\eqref{eq:sg-lyapunov-scaling} with $n=1$, since $d_w=\log5/\log2$ and $d_h=\log3/\log2$. Thus the dependence on $\beta$ may include a factor periodic in $\log\beta$ multiplying the power $\beta^{2d_w/((1+2\alpha)d_w-d_h)}$.
\end{remark}

\subsection{Simulation}\label{sec:gasket-simulation}

Figure~\ref{fig:sg-simulation} shows one realization for $\alpha=0$, $\beta=1.5$ and $u_0\equiv1$, at times $0$, $0.5$, $1$ and $2$. We use a mesh of size $1/16$ on $32K$, with $29\,524$ active vertices, and display the solution on $D=16K$, with $9\,843$ vertices. The two remote corners of $32K$ carry zero Dirichlet conditions; the origin retains the natural reflecting condition.

Let $L$ be the combinatorial graph Laplacian on the active vertices, retaining the diagonal contributions of edges incident to the Dirichlet corners, and let $M=\operatorname{diag}(m_i)$, where $m_i$ is the number of mesh cells incident to vertex $i$, divided by $3^5$. With the normalization $\mu(K)=1$, the graph approximation solves the SDEs
\begin{equation*}
dU_i(t)=-\bigl(M^{-1}(5/3)^4LU(t)\bigr)_i\,dt
+\frac{\beta}{\sqrt{m_i}}U_i(t)\,dB_i(t),
\qquad U_i(0)=1\,,
\end{equation*}
where the $B_i$ are independent standard Brownian motions and $m_i^{-1/2}$ gives the discrete spatial white-noise normalization.  The displayed averages approximate $\bar u_D(t)=\mu(D)^{-1}\int_Du(t,x)\,d\mu(x)$ by vertex quadrature counting only mesh cells contained in $D$. 
\begin{figure}[!htbp]
\centering
\includegraphics[width=0.95\textwidth]{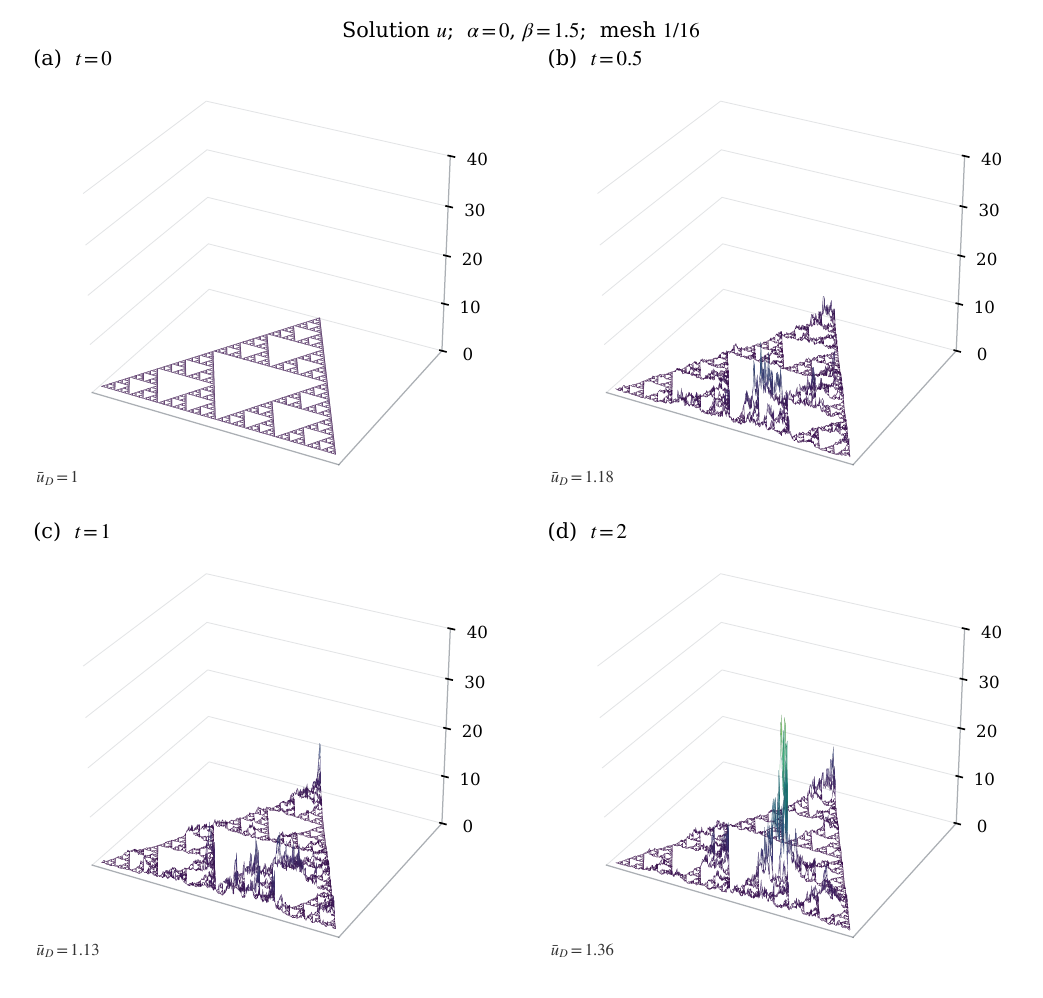}
\caption{Numerical approximation of the solution on the infinite Sierpi\'nski gasket, with space--time white noise ($\alpha=0$), $\beta=1.5$ and initial condition $u_0\equiv1$, observed on the cell $D=16K$ of diameter $16$. The computation uses the larger finite truncation $32K$. All four panels show the same realization; $\bar u_D$ denotes its spatial average over $D$.}
\label{fig:sg-simulation}
\end{figure}
\FloatBarrier

\subsection*{Acknowledgments}
F.B.\ is partially supported by grant 10.46540/4283-00175B from the Independent Research Fund Denmark, by the Villum Investigator grant \emph{Stochastic Analysis in Aarhus}, and by the European Research Council (ERC) under the European Union's Horizon Europe research and innovation programme (RanGe project, Grant Agreement No.~101199772). H.C.\ is partially supported by the Villum Investigator grant \emph{Stochastic Analysis in Aarhus}. L.C.\ and S.K.\ are partially supported by grant 10.46540/4283-00175B from the Independent Research Fund Denmark.
\subsection*{Declaration of generative AI use}

During the final preparation of this manuscript, the authors used Claude
(Fable 5.1) and ChatGPT (Astra) to assist with editorial restructuring,
proofreading, proof presentation, numerical simulations and \LaTeX{} consistency. All the mathematical ideas  are due to the authors with the exception of the following: In Theorem~\ref{thm:lyapunov-main}, the authors initially proved the existence of the moment Lyapunov exponents for integer moments using Feynman-Kac formulas and the theory of Schr\"odinger operators.
ChatGPT (Astra) suggested an approach extending this result to every
real $p\ge1$. The authors then wrote the details and simplified the proof, which turned out to rely on ideas from \cite{BergerChongLacoin2023}. The authors reviewed and edited all AI-assisted material and take full
responsibility for the mathematical statements, proofs, and presentation
of the article.

\bibliographystyle{abbrv}
\bibliography{bibfilerev}
\end{document}